\documentclass[preprint,onefignum,onetabnum]{siamart171218}

\usepackage{lipsum}
\usepackage{amsfonts}
\usepackage{graphicx}
\usepackage{epstopdf}
\usepackage{algorithmic}

\usepackage{amssymb}
\usepackage{amsmath}

\usepackage{todonotes}
\usepackage{subcaption}
\usepackage{makecell}
\usepackage{booktabs}
\usepackage{tabularray}
\UseTblrLibrary{booktabs}
\usepackage{comment}

\providecommand{\norm}[1]{\left\lVert#1\right\rVert}

\newcommand{\weak}{\rightharpoonup}

\newcommand{\be}{\begin{equation}}
\newcommand{\ee}{\end{equation}}
\newcommand{\ba}{\begin{eqnarray}}
\newcommand{\ea}{\end{eqnarray}}
\newcommand{\beq}{\begin{equation}}
\newcommand{\eeq}{\end{equation}}

\renewcommand{\leq}{\leqslant}
\renewcommand{\le}{\leqslant}
\renewcommand{\geq}{\geqslant}
\renewcommand{\ge}{\geqslant}

\def \R {\mathbb{R}}

\def \N {\mathbb{N}}

\def\beq{\begin{equation}}
\def\eeq{\end{equation}}

\def\ba{\begin{array}}
\def\ea{\end{array}}

\usepackage[nameinlink]{cleveref}
\numberwithin{equation}{section}

\newtheorem{rem}{Remark}[section]

\ifpdf
  \DeclareGraphicsExtensions{.eps,.pdf,.png,.jpg}
\else
  \DeclareGraphicsExtensions{.eps}
\fi

\newsiamthm{theo}{Theorem}
\newsiamthm{prop}{Proposition}
\newsiamthm{cor}{Corollary}

\newsiamthm{defi}{Definition}
\newsiamthm{hypo}{Hypothesis}
\newsiamthm{prob}{Method}

\headers{Machine Learning Methods for Feedback Control Laws}{K. Kunisch and D. Vásquez-Varas}

\title{Convergence of Machine Learning Methods for Feedback Control Laws: Averaged Feedback Learning Scheme and Data Driven Methods
}

\author{Karl Kunisch\thanks{Radon Institute for Computational and Applied Mathematics, Austrian Academy of Sciences\\ and Institute of Mathematics and Scientific Computing, University of Graz (\email{karl.kunisch@uni-graz.at}).}
\and Donato Vásquez-Varas\thanks{Radon Institute for Computational and Applied Mathematics, Austrian Academy of Sciences,
  (\email{donato.vasquez-varas@ricam.oeaw.ac.at}.}}

\usepackage{amsopn}

\makeatletter
\newcommand*{\addFileDependency}[1]{
  \typeout{(#1)}
  \@addtofilelist{#1}
  \IfFileExists{#1}{}{\typeout{No file #1.}}
}
\makeatother

\ifpdf
\hypersetup{
  pdftitle={An Example Article},
  pdfauthor={D. Vásquez-Varas, K. Kunisch}
}
\fi

\begin{document}

\maketitle

\begin{abstract}
This work addresses the synthesis of optimal feedback control laws via machine learning. In particular, the \emph{Averaged Feedback Learning Scheme} (AFLS) and a data driven methods are considered.  Hypotheses for each method ensuring the convergence of the evaluation of the objective function of the underlying control problem at the obtained feedback-laws towards the optimal value function are provided. These hypotheses are connected to the regularity of the value function and the stability of the dynamics. In the case of AFLS these hypotheses only require  Hölder continuity of the value function, whereas for the data driven method the value function must be at least $C^2$. It is demonstrated that these methods are connected via their optimality conditions. Additionally, numerical experiments are provided by applying both methods to a family control problems, parameterized by a positive real number which controls the regularity of the value function. For small parameters the value function is smooth and in contrast for  large parameters it is non-differentiable, but semi-concave. The results of the experiments indicate that both methods have a similar performance for the case that the value function is smooth. On the other hand, if the value function is not differentiable, AFLS has a better performance which is consistent with the obtained convergence results.
\end{abstract}

\begin{keywords}
  optimal feedback control, Hamilton-Jacobi-Bellman equation, learning approach,  learning theory, non-linear system, non-smooth value function.
\end{keywords}

\begin{AMS}
49L12,   
49J15,   
49N35, 
68Q32,   	
93B52,   	
\end{AMS}

\section{Introduction}
The problem of constructing feedback-laws via machine learning methods has attracted notable attention. This is due to the fact that classical method for synthesizing feedback laws suffer from the \emph{curse of dimensionality}. Among the well-known methods for solving the HJB equation we mention finite difference schemes \cite{Bonnans}, semi-Lagrangian schemes \cite{Falcone}, and policy iteration \cite{Alla, Beard, Puterman, Santos}.  The curse of dimensionality can  be elevated by machine learning methods, since they are capable of efficiently approximating high dimensional  functions. For instance we can mention the following contributions:  representation formulas \cite{Chow1,Chow2,Chow3,DarbonOsher}, approximating the HJB equation by neural networks \cite{Han,Darbon,Nusken,Onken,Ito,KW1,Ruthotto,chen2023deep,zhao2023offline}, data driven approaches \cite{Nakamura1,Nakamura2,AzKaKK,Kang,Albi,ehring}, max-plus methods \cite{Akian,Gaubert,Dower}, polynomial approximation \cite{Kalise1,Kalise2}, tensor decomposition
methods \cite{Horowitz,Stefansson,Gorodetsky,Dolgov,Oster,Oster2}, POD methods \cite{Alla2,KunischVolk}, tree structure algorithms \cite{Alla1}, and sparse grids techniques\cite{BokanowskiGarckeGriebelPo, Garcke, KangWilcox,BokaWarinProst}, see also the proceedings volume \cite{KaliseKuRa}. Frequently, there is no  proof concerning the question whether the feedback law constructed via machine learning approximates an optimal control, unless the value function is sufficiently smooth  as is the case in \cite{ehring}, where the value functions is supposed to be an element of a reproducing kernel Hilbert space, where an interpolation method is used, or as in \cite{KW1} and \cite{KuVaWal} where the convergence of the method is proved for control problems with $C^2$ value functions.

In the present work  we analyze and study the convergence of two approaches and investigate  the interconnections between them. The first one consists in finding a feedback law minimizing the average of the cost of the control problem with respect to a set of initial conditions. We call this method \emph{Averaged Feedback Learning Scheme} (AFLS). It has been introduced and used in earlier work, see e.g. \cite{KW1,KuVa,KuVa,KuVa2, Onken, Onken2, Ruthotto}, but has not been attributed a name before. In the second approach the feedback  law is obtained by using a regression or interpolation algorithm. We consider two versions of it.  In all cases  the feedback control is parameterized by means of the verification theorem.

The convergence of the methods is investigated  in terms of the distance between the evaluation of the objective function  obtained by the feedback law of the respective method  and the value function of the underlying optimal control problem.
This is  explained in more detail this in Section \ref{sec:Methods}.
We are able to prove the convergence of the AFLS approach under the hypothesis of the existence of a stable sequence of consistent feedback laws, see \cite{KuVa2}.
 This in turn depends on the regularity of the value function. For the data driven approach  we are able to prove its convergence if the value function is at least $C^2$. In order to illustrate the dependence of the convergence in dependence of the regularity  of the value function, we also study numerically  the convergence of the different methods for a family of control problem indexed by a parameter $\gamma>0$. This problem was first given in \cite{KuVa2} where it was proved that its value function is semi-concave for all $\gamma>0$ and is non-differentiable for $\gamma$ large enough. Additionally, in the present article we prove that for $\gamma$ small enough the value function is $C^{\infty}(\R^d)$. Thus applying the methods for this problem for different values of $\gamma$ allows us to observe the influence of the regularity on the behavior of the  proposed methods.

Let us briefly outline the structure of the paper. The problem statement is given in Section 2. Section 3 contains the description of the methods for its solution,  and of the hypotheses needed for the convergence analysis, which is presented in Section 4. Section 5 is devoted to the description of the test problem. It depends on a parameter regulating the regularity of the associated value function. Numerical results are presented in Section 6.  The Appendix contains some technical results.
\section{Control problem}

We consider the task of finding an approximate optimal feedback-law for the following control problem
\begin{equation}
    \min_{u\in L^{2}((0,\infty);\R^m)}\int_{0}^{\infty}\ell(y(t;y_0,u))+\frac{\beta}{2}|u|^2 dt,
    \label{ControlProblem}
\end{equation}
where $y(\cdot ;y_0,u)$ is the solution of the equation
\beq y'=f(y)+Bu,\ y(0)=y_0. \label{ControlProblem:ODE}\eeq
Throughout we assue that $B\in \mathbb{R}^{d\times m}$, that $f$ and $\ell$ are  Lipschitz continuous on bounded subsets of $\mathbb{R}^d$, and $\ell\ge 0$.

A classical way to find   a feedback-law is based on dynamic programming. For this purpose the  value function associated to \eqref{ControlProblem} - \eqref{ControlProblem:ODE} is utilized. It is defined by
\beq V(y_0)= \min_{u\in L^{2}((0,\infty;\R^d))}\int_{0}^{\infty}\left(\ell(y(t;y_0,u))+\frac{\beta}{2}|u(t)|^2 \right)dt\label{def:ValueFunction}.\eeq
In the case that $V$ is $C^{1,1}(\R^d)$ it is known that the feedback law  given by
$$ u(y)=-\frac{1}{\beta}B^{\top}\nabla V(y)\mbox{ for }y\in\R^d$$
is optimal, i.e., for $y_0\in \R^d$, denoting by $y^{*}\in H^{1}_{loc}((0,\infty);\R^{d})$ the unique solution of
$$y'(t)=f(y(t))+B u(y(t)),\ t>0,\ y(0)=y_0$$
and setting $u^{*}(t)=u(y^{*}(t))$ for $t>0$, we have that $u^{*}$ is an optimal solution of \eqref{ControlProblem}. Further, in this case $V$ is a classical solution of the following Hamilton-Jacobi equation
\beq \max_{u\in \R^{m}}\left\{- \nabla V(y)(f(y)+Bu) -\ell(y)-\frac{\beta}{2}|u|^2\right\}=0\mbox{ for all }y\in \R^{d}.\label{HJBeq}\eeq

Hence, for finding an  optimal feedback law it is enough to
solve \eqref{HJBeq}.  But there are two main obstacles for using  this approach directly. The first one relates to the fact that  the value functions is frequently  merely locally Lipschitz continuous, and the second one is that solving \eqref{HJBeq} suffers from the curse of dimensionality. As was mentioned in the introduction, for this reason many authors have tried to overcome this difficulty by using machine learning techniques.

Problem \eqref{ControlProblem}-\eqref{ControlProblem:ODE} is posed over the infinite time horizon, and consequently special attention has to be paid to the well-posedness of this problem. 
Above we have not made a concrete choice for the running cost $\ell$. If $\ell(y) = \frac{1}{2} |y|^2$, then \eqref{ControlProblem}-\eqref{ControlProblem:ODE} relates to optimal stabilization and appropriate stabilizability assumptions are required. Alternatively we can think of a running cost 
containing a discount factor.  These aspects have received much attention in the control literature and we shall not focus on them here. Rather we make assumptions which guarantee that \eqref{ControlProblem}-\eqref{ControlProblem:ODE} admits a control $u$ in feedback form such that the cost in   \eqref{ControlProblem} is finite for all initial conditions $y_0$  in a set $\omega \subset \R^d$. 

\section{Methods for the construction of approximating   optimal feedback laws}
\label{sec:Methods}
For enunciating precisely our goal and main results, we need first some definitions and hypotheses, which were previously  introduced in \cite{KuVa2,KuVaWal,KuVa}. Throughout let $\Omega$ be an open and bounded subset of $\R^{d}$ and let $\omega\subset \R^d$ be another open set with its closure  contained in $\Omega$.
\begin{defi}
\label{defi:Cal}
 For $u\in  W^{1,\infty}_{loc}(\Omega;\R^m)$, $y_0\in \Omega$, and $T\in (0,\infty]$ the function   $y(\cdot ;y_0,u)\in W^{1,\infty}(0,T;\R^d)$  is called the solution to \eqref{ControlProblem:ODE} on $[0,T]$ if it is the only function satisfying
\beq  y'(t)=f(y(t))+Bu(y(t)),\ y(0)=y_0, \ \text{ for } t\in [0,T]. \label{closedloop}\eeq
\end{defi}

Frequently the function $u$ in the above definition in the one associated to  the verification theorem. For this purpose we introduce the following notation.

\begin{defi}
\label{defi:feedback}
    For $v\in C^{1,1}(\Omega)$  the mapping $u_v\in W^{1,\infty}_{loc}(\Omega;\R^{d})$ is defined as
    \beq u_{v}(y)=-\frac{1}{\beta}B^{\top}\nabla v(y)\mbox{ for } y\in \Omega.\eeq
\end{defi}

If $u_v$ is evaluated along a trajectory $y(t)$ of \eqref{ControlProblem:ODE} we refer to it as feedback-law. For $\delta>0$ we further introduce 
\beq \Omega_{\delta}:=\{x\in\Omega: \min_{z\in\partial\Omega}|x-z|>\delta\}.\label{OmegaDelta}\eeq.

\begin{hypo}
\label{hypo:stab}
For the tuple $(T,u,\delta,y_0)$ with $T>0$,  $u\in W^{1,\infty}_{loc}(\Omega;\R^m)$, $\delta>0$, and $y_0\in\Omega_\delta$, the solution $y(\cdot;y_0,u)$ to \eqref{ControlProblem:ODE} exists  on $[0,T]$ and $y(t;y_0,u) \in \Omega_\delta$ for all $t \in [0,T]$.
\end{hypo}

If \Cref{hypo:stab} is assumed for all $T>0$, then it amounts to a stability assumption for \eqref{ControlProblem:ODE} under the control $u$. Under a global Lipschitz condition on $f$ and assumptions on $V$ this assumption was addressed in \cite[Proposition 3.5]{KW1}. It can also be established by Ljapunov function techniques. For this purpose we consider the alternative   \Cref{hypo:Lyapunov} below and in the subsequent remark we discuss that it implies  \Cref{hypo:stab}.

\begin{hypo}\label{hypo:Lyapunov}
 There exist $\tilde \delta>0$, and $w\in C^{1}(\Omega)$ such that for
 $$\omega_{\tilde \delta}:=\{y\in \Omega: w(y) < \sup_{y_0\in \omega} w(y_0)+\tilde\delta\},$$
we have that $ \omega \subset \omega_{\tilde\delta}$,  $\overline{\omega_{\tilde\delta}} \subset \Omega$, and  $\partial \omega_{\tilde\delta}$ is of class $C^1$.
Moreover  $\phi\in C(\Omega)$ is a viscosity super solution of
$$ -\nabla w(y)^{\top}(f(y)+B(y)u_\phi(y))=0 \mbox{ in }\omega_{\tilde\delta},$$
i.e. for every $\bar{y}\in\omega_{\tilde\delta}$ and every $h\in C_{loc}^{1}(\Omega)$ such that $\phi-h$ attains a local minimum at $\bar{y} $ the  following inequality holds
\beq\label{eq:Lap}
\nabla w(\bar{y})^{\top}(f(\bar{y})+B(\bar{y})u_h(\bar{y}))\leq  0.
\eeq
\end{hypo}

\begin{rem}
\label{rem:stab}
Let us consider  $\tilde\delta>0$, $w\in C^{1}$,  and $\phi\in C^{1,1}(\Omega)$ as in    \Cref{hypo:Lyapunov}. Under these conditions $w$ is a classical Ljapunov function, see e.g. \cite{Khalil}, and \eqref{eq:Lap} is satisfied with $u_h=u_\phi$. Consequently we obtain that the trajectories originating from $y_0\in \omega \subset \omega_{\tilde\delta}$ exist for all time $t\ge0$ and cannot escape from $\overline {\omega_{\tilde\delta}}$. Let us set $2 \delta$ as the distance of $\overline{\omega_{\tilde\delta}}$ to $\partial\Omega$, so that $\overline{\omega_{\tilde\delta}}\subset \Omega_\delta$.
Then for all $y_0\in\omega$ and all $T>0$  the tuple $(T,u_{\phi},\delta,y_0)$ satisfies \Cref{hypo:stab}, since ${\omega_{\tilde\delta}}$ is strictly included in $\Omega_\delta$.
\end{rem}

Next we introduce the functional $\mathcal{V}$ which evaluates the cost functional along the trajectory under the control $u$:
\begin{defi}
\label{defi:calV}
Let $(T,u,\delta,y_0)$ be a tuple with $T>0$, $u\in W^{1,\infty}_{loc}(\Omega;\R^m)$, $\delta>0$, and  $y_0\in\Omega_\delta$, and let \Cref{hypo:stab} hold. We define $\mathcal{V}_{T}(y_0,u)$ as
$$ \mathcal{V}_{T}(y_0,u)=\int_{0}^{T}\left(\ell(y(t;y_0,u))+\frac{\beta}{2}|u(y(t;y_0,u))|^2 \right)dt. $$
\end{defi}

Of course we always have $(T,u,\delta,y_0)\ge V(y_0)$, provided the two functionals are well-defined.

We study three methods with the aim of constructing  an approximately optimal feedback law. To explain them we consider a finite time horizon $T>0$, a finite dimensional Banach space $(\Theta,\norm{\cdot}_{\Theta})$, a continuous and coercive function $\mathcal{P}:\Theta\mapsto \R^+$, a penalty coefficient $\alpha>0$, and a function $v:\Theta\mapsto C^{2}(\Omega)$. We will name the tuple $S=((\Theta,\norm{\cdot}_{\Theta}),\mathcal{P},\alpha,v)$  the \textbf{setting} of the learning problems.
We define the methods for the finite time horizon $T$. Subsequently will shall use these methods for a sequence of horizons $T_n$ with $\lim_{n\to\infty}T_n=\infty$ to obtain an approximate feedback law for the infinite horizon problem \eqref{ControlProblem}-\eqref{ControlProblem:ODE}. Note that involving finite horizon problems is natural for numerical realisations. 

The function $v(\theta)$ will play the role of the approximation to the value function $V$ which will henceforth be assumed to exist at least as an element in $L^1(\Omega)$. 

The first method, which we refer to as Averaged Feedback Learning Scheme (AFLS), consists in minimizing the
average cost of \eqref{ControlProblem} plus a penalization term:

\begin{prob}
\label{Problem:LearningProblemOpt}
For a time horizon $T>0$ and a setting $S=((\Theta,\norm{\cdot}_{\Theta}),\mathcal{P},\alpha,v)$, we solve
\beq
\min_{\theta\in \overline{\mathcal{O}_{T,\Theta}}}\frac{1}{|\omega|}\int_{\omega}\mathcal{V}_{T}(y_0,u_{v(\theta)})dy_0+\alpha\mathcal{P}(\theta)\label{LearningProblemOpt}\eeq
where $$ \mathcal{O}_{T,\Theta}:=\{\theta\in\Theta: \exists \delta>0 \mbox{ such that } u_{v_{(\theta)}}\mbox{ satifies \Cref{hypo:stab}}\mbox{ for all }y_0\in\omega\}.$$

\end{prob}

The second method consist in finding $\theta\in \Theta$ such that it minimizes the $L^{2}$ distance with respect to the optimal feedback controls along the trajectories originating from $\omega$:
\begin{prob}
\label{Problem:MachLearningProblem1}
For $V\in C^{1}(\Omega)$, a given time horizon $T>0$, and a setting $S=((\Theta,\norm{\cdot}_{\Theta}),\mathcal{P},\alpha,v)$ we solve
\beq
\min_{\theta\in {\Theta}}\frac{1}{2\beta|\omega|}\int_{\omega}\int_{0}^{T}|B^{\top}(\nabla v(\theta)(y^*(t;y_0))-\nabla V(y^*(t;y_0)))|^2dtdy_0+\alpha\mathcal{P}(\theta),\label{LearningProblemReg}
\eeq
where $y^{*}\in W^{1,\infty}((0,T)\times \omega;\R^d)$ is such that for almost all $y_0\in \omega$, the control given by $u(t)=-\frac{1}{\beta}B^\top\nabla V(y^*(y_0,t))$ is an optimal solution of \eqref{ControlProblem}.
\end{prob}

In a similar way, we consider the problem of minimizing the optimal feedback controls over the whole domain $\Omega$ instead of doing it along the trajectories. This leads to the following formulation:
\begin{prob}
\label{Problem:MachLearningProblem2}
For $V\in W^{1,2}(\Omega)$ and a given setting $S=((\Theta,\norm{\cdot}_{\Theta}),\mathcal{P},\alpha,v)$ we solve
\beq
\min_{\theta\in \Theta}\frac{1}{2\beta|\Omega|}\int_{{\Omega}}|B^{\top}\left(\nabla v(\theta)(y_0)-\nabla V(y_0)\right)|^2dy_0+\alpha \mathcal{P}(\theta).
\label{LearningProblemReg2}
\eeq
\end{prob}

Each of the methods is investigated for a sequence $S_n=((\Theta_n,\norm{\cdot}_{n}),\mathcal{P}_n,\alpha_n,v_n)$  of settings  replacing $S$ in the statements of these methods, with $\Theta_n$ finite dimensional, and a sequence of time horizons $T_n\to \infty$. Appropriate conditions will be established on  $S_n=((\Theta_n,\norm{\cdot}_{n}),\mathcal{P}_n,\alpha_n,v_n)$   to allow an asymptotic analysis as $n \to \infty$.

Observe that \Cref{Problem:LearningProblemOpt} and \Cref{Problem:MachLearningProblem1} when applied to a setting $S_n$ provide a solution   $v_n(\theta_n^*)$ which involves initial conditions in  $\omega$. But   $v_n(\theta_n^*)$ is also defined on all of  $\Omega$ and we hope that it provides  an approximation to the sought value function    also outside of $\omega$. In practice,  of course, initial conditions will not be chosen in all of $\omega$, but only at discrete sample points. These comments refer to the topic of {generalization} in the machine learning context.  For a specific case of a feedback control problems,  this was investigated in \cite[Section 8]{KuVaWal}.

To address the problem of approximating $v$ in Methods \ref{Problem:LearningProblemOpt}-\ref{Problem:MachLearningProblem2} by sequences of finite dimensional settings $S_n=((\Theta_n,\norm{\cdot}_{n}),\mathcal{P}_n,\alpha_n,v_n)$  we require additional hypotheses which are explained next. This will need some patience on the side of the reader.

Let us consider sequences of time horizons $T_n>0$ with $\lim_{n\to \infty}T_n=\infty$, of finite dimensional Bachach spaces $\{(\Theta_{n},\norm{\cdot}_{n}\}_{n\in \N}$,  coercive and continuous penalty operators $\mathcal{P}_{n}:\Theta_n\mapsto\R_+$,  functions  $v_n\in C^1(\Theta_n;C^2(\Omega) )$, and penalty parameters $\alpha_n>0$. Let $\theta_n^{*}\in \Theta_n$ be obtained by solving any of the three problems introduced above.  Our goal then consists in finding assumptions which imply that
\beq \lim_{n\to \infty}\norm{\mathcal{V}_{T_n}(\cdot,u_{v_n(\theta_n^*)})-V}_{L^{p}(\omega)}=0,\label{ConvergenceDef}\eeq
for $p\in [1,\infty]$. 
For this, we will use the following hypotheses:
\begin{hypo}
\label{hypo:approx:vf}
     We have $V\in L^{1}(\Omega)$  and there exist an increasing sequence $T_n$ of time horizons, a constant $\delta\in (0,dist(\omega,\partial\Omega))$, and a sequence $V_n\in C^{2}(\Omega)$ such that $\lim_{n\to \infty}T_{n}=\infty$,  for all $y_0\in \omega$  the tuple  $(T_n,u_{V_{n}},\delta,y_0)$ satisfies \Cref{hypo:stab}   with $T=T_n$, and
    \beq
    \label{hypo:approx:vf:lim}
    \limsup_{n\to \infty}\int_{\omega}(\mathcal V_{\hat{T}_n}(y_0,u_{V_{n}})-V(y_0))dy_0= 0.
    \eeq
\end{hypo}
Note that  \eqref{ConvergenceDef} involves the candidate for an approximately optimal  feedback $u_{v_n(\theta_n^*)}$ while in \eqref{hypo:approx:vf} the feedback $u_{V_{n}}$ involves the approximation $V_n$ of the value function, which is supposed to exist. 
Thus this hypothesis serves the purpose of guaranteeing that $V$ can in principle  be approximated by controls of the form  \eqref{defi:feedback} with $v=V_n$ and utilized in the evaluation of $\mathcal V_{\hat{T}_n}(y_0,u_{V_{n}})$. We shall discuss  the feasibility of  \eqref{hypo:approx:vf} in Remark \ref{remH3.9} right after \Cref{cor:convequivalence} in Section 4. For this purpose we shall require that $V$ is $\alpha-$Hölder continuous with $\alpha>\frac{1}{2}$.

The following hypothesis one will ensure the possibility to approximate these smooth functions  by finitely parameterized  functions, a property which is required for numerical realizations. 
\begin{hypo}
\label{hypo:approx:c2}
For each $g\in C^{1}(\Omega)$ there exist $\theta_n\in\Theta_n$ and a sequence of penalty coefficients $\alpha_n>0$  with $\lim_{n\to \infty} \alpha_n=0$ such that
    \beq \lim_{n\to\infty} v_n(\theta_n)=g\mbox{ in } C_{loc}^1(\Omega) \mbox{ and }\lim_{n\to\infty} \alpha_n\mathcal{P}_n(\theta_n)=0.\eeq
\end{hypo}

Hypotheses \ref{hypo:approx:vf} and \ref{hypo:approx:c2}, and a proper choice of the time horizons $T_{n}$ will imply the convergence of \Cref{Problem:LearningProblemOpt} in the sense of \eqref{ConvergenceDef}.

 A more restrictive version of \Cref{hypo:approx:c2}, requiring higher regularity of the value function,  will allow us to prove that Methods \ref{Problem:MachLearningProblem1} and \ref{Problem:MachLearningProblem2} converge in the same manner as  specified for  \Cref{Problem:LearningProblemOpt} above. We first state this hypothesis and in the following remark we provide an illustrating example. 
\begin{hypo}
\label{hypo:approx:restricc2}
There exists a Banach space $(\Theta,\norm{\cdot }_{\Theta})$   such that $V\in \Theta\cap C^{2}(\Omega)$, $\Theta$ is compactly embedded in $C^{1}(\overline{\Omega})$ and there exists $\delta>0$ such that for all $y_0\in\omega$ the tuple $(T,u_V, \delta,y_0)$ satisfies \Cref{hypo:stab} with $T=\infty$. Further, there exist sequences $\{(\Theta_n,\norm{\cdot }_{n})\}_{n\in\N}$ of finite dimensional Banach spaces, $v_n:\Theta_n\mapsto\Theta$ of continuous functions,  and $P_n:\Theta_n\mapsto\R_+$ of continuous functions, such that for all $g\in \Theta$ there exists a sequence $\theta_{n}\in\Theta_n$ satisfying
$$ \lim_{n\to\infty}\norm{v_n(\theta_n)-g}_{\Theta}=0,\ \sup_{n\in\N}\mathcal{P}_n(\theta_n)<\infty.$$
Moreover  there exist constants $C>0$ and $\sigma>0$ such that  it holds
$$  \norm{v_n(\theta)}_{\Theta}\leq  C\mathcal{P}_n(\theta)^\sigma, \text{ for all } n\in\N \text{ and }\theta\in \Theta_n.$$
\end{hypo}

\begin{rem}
\label{rem:example:approx}
 In order to illustrate the notation and give a situation where \Cref{hypo:approx:restricc2} holds, we consider the case of polynomials as ansatz functions. For the sake of simplicity we take $\Omega=(-1,1)^d$. For $k\in\N$ consider $H^{k}([-1,1])$ with the inner product $\langle\phi,\psi \rangle = \frac{1}{2}\sum_{i=0}^{k} \int_{-1}^{1}\phi^{(i)}(x)\psi^{(i)}(x)dx$ where $\phi^{i}$ stands for the $i$-th derivative of $\phi$. Let $\{\phi^{k}_{i}\}_{i=1}^{\infty}$ be a orthonormal basis of  $H^k((-1,1))$ obtained by applying the Gram-Schmidt process to the set of monomials $\{1,x,x^{2},\ldots\}$. We set $\Theta=H_{mix}^{k}((-1,1)^{d}):=\otimes_{i=1}^{d}H^k((-1,1))$, i.e, $\Theta$ is the tonsorial product of $d$ copies of $H^k((-1,1))$. The space $H_{mix}^{k}((-1,1)^{d})$ is Hilbert with the tensorial inner product. For a multi-index $\alpha\in\N^{d}$ we denote $\phi_{\alpha}(y):=\prod_{i=1}^{d}\phi_i(y_i)^{\alpha_i}$. Hence the set $\{\phi_{\alpha}\}_{\alpha\in\N^d}$ is a basis for $\Theta$. For $n\in\N$ we define $\Lambda_n\subset\N^d$ as the follows:
$$ \Lambda_n=\{\alpha=(\alpha_{1},\ldots,\alpha_d)\in \N^d: \max_{i=1,\ldots,d}\alpha_i \leq n\}.$$
Let us consider an enumeration of $\Lambda_n$ given by $\{\alpha^i\}_{i=1}^{|\Lambda_n|}$.
 We set $\Theta_n=\R^{|\Lambda_n|}$, $v_n(\theta)=\sum_{i=1}^{|\Lambda_n|}\theta_i\phi_{\alpha^i}$, $\norm{\theta}_n=|\theta|$,  and $\mathcal{P}_n(\theta)=\frac{1}{2}|\theta|^2$.  With these choices \Cref{hypo:approx:restricc2}  is met with $\sigma=\frac{1}{2}$ and $k\geq 3$, since $H_{min}^k((-1,1)^d)$ is continuously embedded in $W^{2,\infty}((-1,1)^d)$ (see \cite{SICKEL2009748}).

For high-dimensional problems (i.e. large $d$) the \textbf{hyperbolic-cross} index set is a more efficient choice than $\Lambda_n$. This set of multi-indexes is given  by
$$ \Gamma_n=\{\alpha=(\alpha_{1},\ldots,\alpha_d)\in N^d: \prod_{i=1}^{d}(\alpha_i+1)\leq n\}.$$

For $\Omega=\otimes_{i=1}^{d}(a_i,b_i)$ with $a_i<b_i$ for $i\in\{1,\ldots,d\}$ a change of variables is needed, that is, for $x\in\Omega$ we define $\tilde{\phi}_{\alpha}(x)=\prod_{i=1}^{d}\phi_{\alpha_i}\left(\frac{2}{(b_i-a_i)}\left(x_i-\frac{a_i+b_i}{2}\right)\right)$ and for $\theta\in\R^{|\Lambda_n|}$ we set $\tilde{v}_n(\theta)=\sum_{i=1}^{n}\theta_i\tilde{\phi}_{\alpha^i}$. With reference to \Cref{hypo:approx:restricc2}  we choose  $\Theta_n= \R^{|\Lambda_n|}$ and $\Theta = H^k_{mix} ((-1,1)^d)$ as above. Again the embedding of $\Theta$ into $C^1(\bar \Omega)$ is compact for $k \ge 3$.

In a more general perspective, \Cref{hypo:approx:restricc2} also encsompasses other interesting cases  such as  neural networks with $\Theta$ being the respective Barron space provided the activation function is smooth, or  reproducing kernel Hilbert spaces with a $C^2$ kernel and $\Theta$ as the corresponding native space. In such cases $\Theta$  can then be chosen as the space of limits  of functions of the form $v_n(\theta_n)$ for some $\theta_n\in\Theta_n$, endowed with the proper norm.
\end{rem}



\begin{rem}
\label{rem:reg}
Note that in contrast to  \Cref{Problem:LearningProblemOpt}, knowledge of the value function is required for \Cref{Problem:MachLearningProblem1} and \Cref{Problem:MachLearningProblem2}. In practice this can be realized by approximating the integrals by finite sums or by Monte Carlo methods with values for $\nabla V$ at sample points obtained by open loop control solves. This approach was investigated in \cite{AzKaKK} for instance. Recall the $\nabla V$ is related to the adjoint state associated to the control problem \eqref{closedloop}-\eqref{ControlProblem:ODE}, see \cite{AzKaKK}. It is therefore obtained  as a by-product in every numerical approach for open loop optimal control.  But, as already pointed out above, approximating the feedback over the smaller set $\omega$ by globally defined functions also serves the purpose to utilize the approximating feedback laws outside of $\omega$. This \emph{generalization step}, of course, requires additional analysis and will depend on the underlying dynamics characterized by $f$.

 Clearly \Cref{Problem:MachLearningProblem1} and \Cref{Problem:MachLearningProblem2} are more closely  related to each other,  than either of them to  \Cref{Problem:LearningProblemOpt}.  In practice data generation will be cheaper for \Cref{Problem:MachLearningProblem1}   than for \Cref{Problem:MachLearningProblem2} since for the former, information is gathered all along the trajectory,  and hence in  the context of sampling fewer data points (initial conditions) will be necessary for  \Cref{Problem:MachLearningProblem1} than for \Cref{Problem:MachLearningProblem2}.
Analytically \Cref{Problem:MachLearningProblem1} requires $C^1$-regularity $V$ and hence it required more regularity of the minimal value function  than is needed for \Cref{Problem:MachLearningProblem2}. 

\end{rem}

\begin{rem}
\label{rem:regressionproblems}
         Methods \ref{Problem:MachLearningProblem1} and \ref{Problem:MachLearningProblem2} minimize the $L^{2}$ distance with respect to the optimal feedback-law. While  the first one explicitly involves information along all of the optimal trajectories, the second one utilizes the whole domain $\Omega$.
         Next we provide an interesting relation between these two formulations.
         Namely, assuming that $\frac{1}{|det(D_{y_0}y^{*}(\cdot,\cdot))|}\in L^{\infty}((0,T)\times\omega)$,  setting $\Gamma_{T}=\{z\in\R^d: \exists (t,y_0)\in [0,T]\times\omega, y^{*}(t,y_0)=z\}$, and defining ${{y^0}}(t,z)$ for $(t,z)\in[0,T]\times\Gamma_T$ as the inverse function of $y^{*}$ with respect to its second argument, we obtain by means of the change of variables $z=y^{*}(t,y_0)$ that
     \begin{equation*}
     \begin{array}{l}
     \displaystyle\int_{\omega}\int_0^T |B^{\top}(\nabla v(\theta)(y^*(t,y_0))-\nabla V(y^*(t,y_0)))|^2dtdy_0\\[1.5ex]\displaystyle
    \qquad=\int_{\Gamma_{T}}|B^{\top}(\nabla v(\theta)(z)-\nabla V(z))|^2g(z)dz
      \end{array}
    \end{equation*}
    where $$ g(z)=\int_{\{t\in [0,T]:\ \exists y_0\in\omega,\ y^{*}(t;y_0)=z\}} \frac{1}{|det(D_{y_0}y^{*}(t,{{y^0}}(t,z)))|} dt,$$
    where $D_{y_0}y^{*}(\cdot,\cdot)$ denotes the derivative of $y^{*}$  with respect to its second argument.  This reflects that  the main difference between Problems \ref{Problem:MachLearningProblem1} and \ref{Problem:MachLearningProblem2}  consists in how they weigh the initial conditions. We also note that as consequence of classical ODE theory  the hypotheses on ${det(D_{y_0}y^{*})}$ expressed above hold if $V\in C^{2}(\Omega)$ .
\end{rem}

\section{Convergence analysis}

This section is devoted to the convergence analysis for Methods \ref{Problem:LearningProblemOpt} - \ref{Problem:MachLearningProblem2}. We first state the two main results and then provide their proofs below.  For  \Cref{Problem:LearningProblemOpt} we have the following convergence property.

\begin{theo}
\label{theo:conv:LearningProblemOpt}
Let Hypotheses \ref{hypo:approx:vf} and \ref{hypo:approx:c2} hold for sequences $(\Theta_n,\norm{\cdot}_{n})$ , $\mathcal{P}_n$, $v_n$ and $T_n$. Then there exist a sequence $\alpha_{n}\!>\!0$ and a sub-sequence $((\Theta_{k_n},\norm{\cdot}_{k_n}),\mathcal{P}_{k_n},v_{k_n})\!$ for which  \Cref{Problem:LearningProblemOpt}  with setting $S_n=((\Theta_{k_n},\norm{\cdot}_{k_n}),\mathcal{P}_{k_n},\alpha_n,v_{k_n})$ has at least one solution. Further, for every sequence of solutions $\theta^{*}_n\in\Theta_{k_n}$ to \Cref{Problem:LearningProblemOpt}  with $S=S_n$ we have
$$ \lim_{n\to\infty}\mathcal{V}_{T_n}(\cdot,u_{v_{k_n}(\theta_{n}^*)})=V\mbox{ in }L^{1}(\omega) \mbox{ and } \lim_{n\to\infty}\alpha_n\mathcal{P}_{k_n}(\theta_n^*)=0.$$
\end{theo}
\begin{rem}
In the case of a nested setting, that is, if $\Theta_n\subset\Theta_{n+1}$ and $\mathcal{P}_{n+1}(\theta)=\mathcal{P}_{n}(\theta)$ for $\theta\in\Theta_n$, it is possible to get the convergence result for the whole sequence of settings and not only for a sub-sequence of them by associating the time horizons and penalty coefficients in an appropriate manner. For instance, for $j\in \N$ let us consider $n_j$ the smallest $n\in\N$ such that $k_n\leq j\leq k_{k+1}$. Then the sequence of settings $\tilde{S}_{j}=((\Theta_{j},\norm{\cdot}_{j}\mathcal{P}_{j},\alpha_{n_j},v_{j})$ and the sequence of times horizons $\tilde{T}_{j}=T_{n_j}$ satisfy
$$ \lim_{j\to\infty}\mathcal{V}_{\tilde{T}_j}(\cdot,u_{v_n(\tilde{\theta}_j*)})=V\mbox{ in }L^{1}(\omega) \mbox{ and } \lim_{n\to\infty}\alpha_{n_j}\mathcal{P}_j(\tilde{\theta}_j^*)=0,$$
where $\tilde{\theta}_j^*\in\Theta_j$ is the solution to \Cref{Problem:LearningProblemOpt} with $S=\tilde{S}_j$.
\end{rem}

In order to prove the convergence for  \Cref{Problem:MachLearningProblem2}, we shall need to assume \Cref{hypo:approx:restricc2} instead of   Hypotheses \ref{hypo:approx:vf} and \ref{hypo:approx:c2}. In this case, the convergence holds in $L^{\infty}(\omega)$.  The regularity assumption on $V$ expressed in Hypothesis \ref{hypo:approx:restricc2}, however,  is more restrictive  than the hypotheses for Theorem \ref{theo:conv:LearningProblemOpt} where the assumptions on $V$ are only those involved in Hypothesis \ref{hypo:approx:vf} . We shall see that the convergence for  \Cref{Problem:MachLearningProblem2} also implies the convergence for  \Cref{Problem:MachLearningProblem1}, by using a change of variables as was already mentioned in \Cref{rem:regressionproblems}.

%

\begin{theo}
\label{theo:conv:LearningProblemReg}
Let $V$, $(\Theta,\norm{\cdot}_{\Theta})$, and a sequence   $S_n=((\Theta_n,\norm{\cdot}_{n}),\mathcal{P}_n,\alpha_n,v_n)$ be such that \Cref{hypo:approx:restricc2} holds. Further assume that $\lim_{n\to\infty}\alpha_n=0$ and there exists a constant $C>0$ such that $\norm{B^\top(\nabla v_n(\theta_n)-V)}_{L^2(\Omega;\R^m)}^2\leq C \alpha_n$ where $\theta_n\in\Theta_n$ is as in \Cref{hypo:approx:restricc2} with $g=V$. Then there exists a sequence of times $T_{n}>0$  for which every sequence $\theta^{*}_n\in\overline{\mathcal{O}_{T_n}}$ of solutions of either  \Cref{Problem:MachLearningProblem1} or \Cref{Problem:MachLearningProblem2} with $S=S_n$ and $T=T_n$ satisfies that $\theta_n^{*}\in \mathcal{O}_{T_n}$ and
$$ \lim_{n\to\infty}\mathcal{V}_{T_n}( \cdot,u_{v_n(\theta_n^*)})=V\mbox{ in }L^{\infty}(\omega). $$
\end{theo}

To verify these two results we start by proving the equivalence between \eqref{hypo:approx:vf:lim},
and the following two properties
\beq
\lim_{n\to\infty}\mathcal{V}_{\hat{T}_n}(\cdot,u_{V_n})=V\mbox{ in }L^{1}(\omega),
\label{hypo:approx:vf:lim2}
\eeq
and
\beq
\lim_{n\to\infty}\int_{\omega}(\mathcal{V}_{\hat{T}_n}(y_0 ,u_{V_n})-V(y_0))^{+}dy_0=0.
\label{hypo:approx:vf:lim3}
\eeq

\begin{lemma}
\label{lemma:NegativePartEverywhereConv}
    Let $T_n>0$ be an increasing sequence with $\lim_{n\to \infty}T_{n}=\infty$, let $\delta$ be a positive constant,  and let  $V_n\in C^{2}(\overline{\Omega})$ be as sequence, such that the tuple
     $(T_n,u_{V_{n}},\delta,y_0)$ satisfies
     \Cref{hypo:stab} for all $y_0\in \omega$ and $n\in\N$. Then for all $y_0\in \omega$ the following identity holds
    \beq \lim_{n\to \infty}(V(y_0)-\mathcal{V}_{T_{n}}(y_0,u_{V_{n}}))^{+}=0 .\label{lemma:NegativePartEverywhereConv:lim}\eeq
\end{lemma}

\begin{proof}
    By contradiction, let us assume that there exists $y_0$ such that \eqref{lemma:NegativePartEverywhereConv:lim} does not hold. This implies that there exist $\kappa>0$ and a sub-sequence still denoted by $V_{n}$ such that
    \beq (V(y_0)-\mathcal{V}_{T_{n}}(y_0,u_{V_n}))^{+}>\kappa
    \label{lemma:NegativePartEverywhereConv:proof:lower}
    \eeq
    for all $n\in\N$. Here $r^+$ stands for the positive part of $r\in \mathbb{R}$.  Let us denote $y_{n}=y(t;y_0,u_{V_{n}})$ and $u_{n}=u_{V_n}\circ y_{n}=-\frac{1}{\beta}B^\top\nabla V_n(y_n)$. By \Cref{hypo:stab} and the boundedness of $\Omega$  we have that $\{\|y_{n}\|_{L^{\infty}((0,T_{n});\R^{d})}\}_{n=1}^{\infty}$ is bounded. By \eqref{lemma:NegativePartEverywhereConv:proof:lower} and the definition of $u_n$ we have that for all $n\in\N$
    \beq  \int_{0}^{T_{n}}\left(\ell (y_{n}(t))+\frac{\beta}{2}|u_{n}(t)|^2\right)dt\leq V(y_0)-\kappa.\label{lemma:NegativePartEverywhereConv:proof:Vesp}\eeq
    This implies that $\{\|u_{n}\|_{L^{2}((0,T_{n});\R^m)}\}_{n=1}^{\infty}$ is bounded. This, together with the fact that  $y_{n}$ is a solution of
    \beq \frac{d}{dt}y_{n}(t)=f(y_{n}(t))+Bu_{n}(t),\ t\in (0,T_{n}), \ y_{n}(0)=y_0,\label{lemma:NegativePartEverywhereConv:proof:yepseq}\eeq
    that $f$ is Lipschitz continuous on bounded sets, and that  $\{y_n\}_{n=1}^\infty$ is bounded in $L^{\infty}((0,T_{n});\R^{d})$, we obtain that $\{y_{n}\}_{n=1}^\infty$ is uniformly bounded in $H^{1}((0,T_{n});\R^{d})$. Therefore, by a diagonalization argument, with respect to $\{T_n\}_{n\in\N}$ and $\{(y_n,u_n)\}_{n\in\N}$,  there exists $y^{*}\in H_{loc}^{1}((0,\infty);\R^{d})$ and $u\in L_{loc}^{2}((0,\infty);\R^{m})$ such that by passing to a sub-sequence, if  needed, we have
    \begin{equation}
    \begin{array}l
    u_{n}\weak u^{*}\mbox{ in }L_{loc}^{2}((0,\infty);\R^{m}), \ y_{n}\weak y^{*}\mbox{ in }H_{loc}^{1}((0,\infty);\R^{d}), \\[1.5ex]
    \mbox{and }y_{n}\to y^{*}\mbox{ in }C_{loc}([0,\infty);\R^{d}),
    \label{lemma:NegativePartEverywhereConv:proof:convergences}
    \end{array}
    \end{equation}
    as $n$ tends to infinity. Using these facts together with \eqref{lemma:NegativePartEverywhereConv:proof:yepseq} we have that
    \beq \frac{d}{dt}y^{*}(t)=f(y^{*}(t))+Bu^{*}(t),\ t\in (0,\infty), \ y^{*}(0)=y_0.\label{lemma:NegativePartEverywhereConv:proof:yeq}\eeq
    Next, let $T>0$ be a fixed time horizon and choose $n_0\in\N$ such that for all $n>n_0$ we have $T<T_{n}$. Hence by \eqref{lemma:NegativePartEverywhereConv:proof:Vesp} we get that for all $n>n_0$
    \beq  \int_{0}^{T}\left(\ell (y_{n}(t))+\frac{\beta}{2}|u_{n}(t)|^2\right)dt\leq V(y_0)-\kappa.\label{lemma:NegativePartEverywhereConv:proof:VepsT}\eeq
    Using \eqref{lemma:NegativePartEverywhereConv:proof:convergences}, the continuity of $\ell$ and the weak lower semi-continuity of the norm of $L^{2}((0,T);\R^{m})$,  letting $n$ tend to infinity in \eqref{lemma:NegativePartEverywhereConv:proof:VepsT} we obtain
    \beq  
    \int_{0}^{T}\left(\ell (y^*(t))+\frac{\beta}{2}|u^*(t)|^2\right)dt\leq V(y_0)-\kappa.
    \label{lemma:NegativePartEverywhereConv:proof:VT}
    \eeq
    Since this holds for any arbitrary $T$ we get that $u^{*}\in L^{2}((0,\infty);\R^{d})$ and
    $$ \int_{0}^{\infty}\left(\ell (y^*(t))+\frac{\beta}{2}|u^*(t)|^2\right)dt\leq V(y_0)-\kappa,$$
    which contradicts the fact that $V$ is the value function. Consequently, we have that for all $y_0\in \omega$
    \beq \lim_{n\to \infty}(V(y_0)-\mathcal{V}_{T_{n}}(y_0,u_{V_{n}}))^{+}=0.
    \eeq
    
\end{proof}

\begin{cor}
\label{cor:convequivalence}
   Under the hypotheses of the previous lemma
   \eqref{hypo:approx:vf:lim}, \eqref{hypo:approx:vf:lim2} and \eqref{hypo:approx:vf:lim3} are equivalent.
\end{cor}
\begin{proof}
  Since $(V(y_0)-\mathcal{V}_{T_{n}}(y_0,u_{V_n}))^{+}\leq V(y_0)$ we can apply the dominated convergence theorem together with \eqref{lemma:NegativePartEverywhereConv:lim} to get that
    \beq \lim_{n\to\infty}\int_{\omega}(V(y_0)-\mathcal{V}_{T_{n}}(y_0,u_{V_n}))^{+}dy_0=0.\label{lemma:convequivalence:proof:pospartconv}\eeq
    Rewriting the $L^{1}(\omega)$ norm of $V(y_0)-\mathcal{V}_{T_{n}}(y_0,u_{V_n})$ we find that
    \begin{equation*}
    \begin{array}{l}
\displaystyle    \int_{\omega}|V(y_0)-\mathcal{V}_{T_{n}}(y_0,u_{V_n})|dy_0\\[1.5ex]
\displaystyle    =\int_{\omega}(\mathcal{V}_{T_{n}}(y_0,u_{V_n}))-V(y_0))^{+}dy_0+\int_{\omega}(V(y_0)-\mathcal{V}_{T_{n}}(y_0,u_{V_n}))^{+}dy_0.$$
    \end{array}
    \end{equation*}
    Since the  convergence to 0 of the last term in the previous inequality is always granted under the hypotheses, we get the equivalence of \eqref{hypo:approx:vf:lim2} and \eqref{hypo:approx:vf:lim3}. Further, condition \eqref{hypo:approx:vf:lim2} clearly implies \eqref{hypo:approx:vf:lim}. To conclude we only need to prove that \eqref{hypo:approx:vf:lim} implies either \eqref{hypo:approx:vf:lim2} or \eqref{hypo:approx:vf:lim3}. For this purpose we note that
     \begin{equation*}
     \begin{array}{ll}
\displaystyle     0 & \displaystyle\le \liminf \int_{\omega}(\mathcal{V}_{T_{n}}(y_0,u_{V_n})-V(y_0))^{+}dy_0\\[1.5ex] \displaystyle
     & \displaystyle\le \lim  \int_{\omega}(V(y_0)-\mathcal{V}_{T_{n}}(y_0,u_{V_n}))^{+}dy_0
     + \limsup \int_{\omega}(\mathcal{V}_{T_{n}}(y_0,u_{V_n})-V(y_0))dy_0.
     \end{array}
     \end{equation*}
    Hence, if \eqref{hypo:approx:vf:lim} holds, we can combine it with \eqref{lemma:convequivalence:proof:pospartconv} to get \eqref{hypo:approx:vf:lim3}.
\end{proof}

We are now prepared to discuss the feasibility of \Cref{hypo:approx:vf}.

\begin{rem}\label{remH3.9}
Concerning the feasibility of  \Cref{hypo:approx:vf}, we assume that \Cref{hypo:Lyapunov} holds  and that
$V$ is $\alpha-$Hölder continuous with $\alpha>\frac{1}{2}$. Then we can use  \cite[Lemma 7.1c, Theorem 5.4]{KuVa2}, with $\phi=V \in C(\Omega)$, and  choose $\delta>0$,   to  construct sequences $\{u_{V_{n}}\}_{n\in\N}$  in $C^2(\Omega)$ and $\{T_n\}_{n\in\N}$ with $\lim_{n\to \infty}T_n= \infty$  such that   the tuple  $(T_n,u_{V_n},\delta,y_0)$ satisfies \Cref{hypo:stab} for all $y_0\in  \omega$ and all $n$. This is the first requirement in  \Cref{hypo:approx:vf}.  
 
Regarding the second part, using again \cite[Theorem 5.4]{KuVa2} the tuple  $(T_n,u_{V_n},\delta,y_0)$ can be chosen such that 
\beq 
0\le \int_{\omega}\big(\mathcal{V}_{T_{n}}(\cdot,u_{V_n})+V(y(T_n,u_{V_n};\cdot) - V(\cdot) \big) dy_0  \to 0 \text{ for } n \to \infty.
\eeq
Since $V(\cdot) \ge 0$ this implies that 
$0\le \int_{\omega}\big(\mathcal{V}_{T_{n}}(\cdot,u_{V_n}) - V(\cdot) \big)^+ dy_0  \to 0 \text{ for } n \to \infty.$
This together with \eqref{lemma:convequivalence:proof:pospartconv} implies that 
   $ \lim_{n\to \infty}\int_{\omega}(\mathcal V_{{T}_n}(y_0,u_{V_{n}})-V(y_0))dy_0= 0,$
   as required.
 \end{rem}

Before continuing we need to establish the following technical result concerning the stability of the
solution to \eqref{closedloop} if  $u$  is of the form \eqref{defi:feedback} and perturbations occur in $v$. Its proof uses Lemma  \ref{lemma:LipschitzContCalV}.

\begin{cor}
    \label{cor:Stab:c2}
    Let $T>0$ be a time horizon, $y_0\in\omega$, $g_1\in C^{2}(\Omega)$, $g_2\in C^{2}(\Omega)$ and $\delta>0$, be such that $(T,u_{g_2},\delta,y_0)$ satisfies \Cref{hypo:stab}. Assume that
    \beq
        \frac{|B|\norm{u_{g_1}-u_{g_2}}_{L^\infty(\Omega_{\frac{\delta}{4}};\R^m)}}{a}\left(e^{Ta}-1\right)\leq \frac{\delta}{2},
        \label{cor:Stab:c2:error}
    \eeq
    where
    \beq a=\norm{Df}_{L^{\infty}(\Omega_{\frac{\delta}{4}};\R^{d\times d})}+\frac{|B|^2}{\beta}\norm{\nabla^2 g_2}_{L^{\infty}(\Omega_{\frac{\delta}{4}};\R^{d\times d})},\label{cor:Stab:c2:a}\eeq 
    and $Df$ stands for the derivative of $f$. 
    Then $(T,u_{g_1},\frac{\delta}{2},y_0)$ satisfies \Cref{hypo:stab}.
\end{cor}
\begin{proof}
    For abbreviation we set $y_1(t)=y(t;y_0,u_{g_1})$ and $y_2(t)=y(t;y_0,u_{g_2})$ for all $t\geq 0$ such that  these expressions are well defined. Let $\hat{T}>0$ be the escape time of $y_1$ from $\Omega_{\frac{\delta}{2}}$, i.e. $\hat{T}=\sup\{\tilde{T}>0:y_1(t)\in \Omega_{\frac{\delta}{2}} \mbox{ for all }t\in [0,\tilde{T}]\}$. We  need to prove that $\hat{T}\geq T$. By contradiction,  assume that $\hat{T}<T$. Then there exists $\Tilde{T} \in (\hat{T},T)$ such that for all $t\in (\hat{T},\tilde{T}]$ we have $y_1(t)\in\Omega_{\frac{\delta}{4}}\setminus\Omega_{\frac{\delta}{2}}$. Combining
    \eqref{lemma:LipschitzContCalV:state} (with $\delta$ replaced by $\frac{\delta}{4}$)  from the Appendix, which establishes the Lipschitz continuous dependence of $y$ on $u_g$,  and \eqref{cor:Stab:c2:error} we obtain for all $ t\in (\hat{T},\tilde{T}]$ that
    $$|y_1(t)-y_2(t)|\leq \frac{\delta}{2}.$$
Here we use that if  $(T,u_{g_2},\delta,y_0)$ satisfies \Cref{hypo:stab} then it is also satisfied with  $(T,u_{g_2},\frac{\delta}{4},y_0)$. Further again using that $(T,u_{g_2},\delta,y_0)$ satisfies \Cref{hypo:stab} and that $\Tilde{T}<T$ we get for all $t\in (\hat{T},\tilde{T}]$ and for all $z\in \partial\Omega$
    $$ \delta\leq |z-y_2(t)|< |z-y_1(t)|+|y_1(t)-y_2(t)|\leq \frac{\delta}{2}+|z-y_1(t)|.$$
    This implies that for $t\in (\hat{T},\tilde{T}]$ we have $|z-y_1(t)|> \frac{\delta}{2}$ for all $z\in \partial\Omega$ which implies that $y_1(t)\in \Omega_{\frac{\delta}{2}}$ for $t\in (\hat{T},\tilde{T}]$. This contradicts the definition of $\hat{T}$ and therefore $\hat{T}\geq T$.
\end{proof}

\begin{proof}[Proof of Theorem \ref{theo:conv:LearningProblemOpt}]
    For proving the existence of a sub-sequences of $\alpha_n$, $T_n$ and $((\Theta_n,\norm{\cdot}_n),\mathcal{P}_n,v_n)$, such that  \Cref{Problem:LearningProblemOpt} has a solution we proceed similarly to the proof of Lemma 3 in \cite{KuVaWal}. Let $T_n$, $ V_n$ and  $\delta$ be those appearing in \Cref{hypo:approx:vf}. For $n\in\N$ fixed  we use \Cref{hypo:approx:c2} with $g=V_n$ to define $\theta_{k,n} \in \Theta_k$ and $\alpha_{k,n}>0$ such that $  \lim_{k\to\infty} v_{k}(\theta_{k,n})=V_n\mbox{ in } C_{loc}^1(\Omega)$ and $\lim_{k\to\infty} \alpha_{k,n}\mathcal{P}_k(\theta_{k,n})=0$.

    Let $y_0\in\omega$ be fixed and set $u_n(x)=u_{V_n}(x)$ and $u_{k,n}(x)=u_{v_k(\theta_{k,n})}(x)$ for all $x\in\Omega$ and $n,k\in\N$. Let  $L_n=\norm{\nabla^2 V_n}_{L^{\infty}\big(\Omega_{\frac{\delta}{4}};\R^{d\times d}\big)}$. Since for each $n\in\N$, $v_{k}(\theta_{k,n})$ converges to $V_n$ in $C_{loc}^{1}(\Omega)$,  there exists $k(n)$ such that for all $k>k(n)$ we have
    \beq
    \begin{array}l
    \left(e^{(L_f+\frac{|B|^2}{\beta}L_{n})T_n}-1\right)\frac{|B|\|u^{k,n}- u^{n}\|_{L^{\infty}(\Omega_{\frac{\delta}{4}};\R^m)}}{(L_f+\frac{|B|^2}{\beta}L_{n})}\leq \frac{\delta}{2}\mbox{ and }\norm{u^{k,n}- u^{n}}_{L^{\infty}(\Omega_{\frac{\delta}{4}};\R^m)}\leq 1.
    \label{theo:conv:LearningProblemOpt:ineq:delta2}
    \end{array}
    \eeq
    In view of \eqref{theo:conv:LearningProblemOpt:ineq:delta2} and the fact that $(T_n,u_n,\delta,y_0)$ satisfies \Cref{hypo:stab} we can apply \Cref{cor:Stab:c2} with $g_1=v_{k}(\theta_{k,n})$, $g_2=V_n$ and $T=T_n$ to obtain that for all $y_0\in\omega$ and $k\geq k(n)$ that  the tuple $(T_n,u^{k,n},\frac{\delta}{2},y_0)$ satisfies \Cref{hypo:stab}. Consequently, we have that $\theta_{k,n}\in\mathcal{O}_{T_n,\Theta_k}$ for $k>k(n)$. Combining this, the fact that $\mathcal{P}_{k}$ is coercive and continuous, and that $\mathcal{V}_{T_{n}}$ is continuous, we get  by the direct method that  \Cref{Problem:LearningProblemOpt} with $S=((\Theta_k,\norm{\cdot}_k),\mathcal{P}_{k},\alpha_{k,n},v_k)$ and $T=T_n$ for $k>k(n)$ has an optimal solution.

For the second part we use \eqref{lemma:LipschitzContCalV:value}. With this in mind we set
 $$K_n=|B| \,\big[\norm{\nabla\ell}_{L^{\infty}(\Omega;\R^{d})}(e^{(L_f+\frac{|B|^2}{\beta}L_{n})T_{n}}-1)\frac{1}
   {(L_f+\frac{|B|^2}{\beta}L_{n})^2}+\frac{  T_n \beta}{2}(\norm{u^{n}}_{L^{\infty}(\Omega_{\frac{\delta}{2}})}+1)\big].$$
    Using \eqref{lemma:LipschitzContCalV:value} we obtain  for all $n\in \N$ and $k>k(n)$
    \beq \int_{\omega}\mathcal{V}_{T_n}(y_0,v_k(\theta_{k,n})) dy_0\leq \int_{\omega}\mathcal{V}_{T_{n}}(y_0,V_n)dy_0+|\omega|K_n\norm{u^{k,n}-u^{n}}_{L^{\infty}(\Omega_{\frac{\delta}{4}};\R^{m})}
    \label{theo:conv:LearningProblemOpt:vknbound}\eeq
    By \Cref{hypo:approx:c2} we can choose $k_n\ge k(n)$ with $\lim_{n\to \infty} k_n= \infty$,  such that
    \beq \lim_{n\to \infty}|\omega|K_n\norm{u^{k_n,n}-u^{n}}_{L^{\infty}(\Omega_{\frac{\delta}{4}};\R^{m})}+\alpha_{k_n,n}\mathcal{P}_{k_n}(\theta_{k_n,n})=0.\label{theo:conv:LearningProblemOpt:subseq}\eeq
    Let  $\theta_{n}^{*}\in\Theta_{k_n}$ be an optimal solution of \Cref{Problem:LearningProblemOpt} for the choices $T=T_n$ and $S=((\Theta_{k_n},\norm{\cdot}_{k_n}),\mathcal{P}_{k},\alpha_{k_n,n},v_{k_n})$. Then by \eqref{theo:conv:LearningProblemOpt:vknbound}
     $$\begin{array}{l}
          \displaystyle \int_{\omega}\mathcal{V}_{T_n}(y_0,v_{k_n}(\theta^{*}_{k_n})) dy_0 \!+\!\alpha_{k_n}\mathcal{P}_{k_n}(\theta_{k_n})\!\leq\! \int_{\omega}\mathcal{V}_{T_n}(y_0,v_{k_n}(\theta_{k_n,n})) dy_0 +\alpha_{k_n}\mathcal{P}_{k_n}(\theta_{k_n,n})\\
          \displaystyle \leq  \int_{\omega}\mathcal{V}_{T_n}(y_0,u_{V_n}) dy_0+\alpha_{k_n}\mathcal{P}_{k_n}(\theta_{k_n,n})+|\omega|K_n\norm{u^{k_n,n}-u^{n}}_{L^{\infty}(\Omega_{\frac{\delta}{4}};
     \R^{m})}.
     \end{array}$$
     Combining the above inequality, \eqref{theo:conv:LearningProblemOpt:subseq}  , and \Cref{hypo:approx:vf}, we arrive at
     $$ \limsup_{n\to\infty}\int_{\omega}\mathcal{V}_{T_n}(y_0,v_{k_n}(\theta^*_{k_n})) dy_0 \leq \int_{\omega}V(y_0) dy_0.$$
     Therefore, by \Cref{cor:convequivalence} we conclude the proof.
\end{proof}
\begin{proof}[Proof of Theorem \ref{theo:conv:LearningProblemReg}]
Let us start by considering $\theta^{*}_n$ as  solution of \Cref{Problem:MachLearningProblem2}. From the optimality  of $\theta^{*}_n$ we obtain that
\beq
\begin{array}l
\frac{1}{2\beta}\norm{B^{\top}(\nabla v_n(\theta^{*}_n)-\nabla V)}_{L^{2}(\Omega)}^2+\alpha_n\mathcal{P}_n(\theta_n^*) \\[1.5ex]
\leq \frac{1}{2\beta}\norm{B^{\top}(\nabla v_n(\theta_n)-\nabla V)}_{L^{2}(\Omega)}^2+\alpha_n\mathcal{P}_n(\theta_n)
\end{array}
\eeq
with $\theta_n$ as in  \Cref{hypo:approx:restricc2}.  By \Cref{hypo:approx:restricc2} we know that the right hand side of this inequality converges to $0$ as $n$ goes to infinity and consequently we obtain that
$$\lim_{n\to\infty}B^{\top}\nabla v_n(\theta^{*}_n)=B^{\top}\nabla V\mbox{ in }L^{2}(\Omega). $$
Moreover,  again by \Cref{hypo:approx:restricc2}  and the optimality of $\theta_{n}^{*}$ we obtain that $\norm{v_{n}(\theta_n^{*})}_{\Theta}$ is uniformly bounded for $n\in\N$. This together with the fact that $\Theta$ is compactly embedded in $C^{1}(\overline{\Omega})$ implies that $B^{\top}\nabla v_{n}(\theta_n^{*})$ converges to $B^{\top}\nabla V$ in $C(\overline{\Omega})$. Let $T_{n}>0$ be a sequence such that $\lim_{n\to\infty}e^{T_n a}\norm{B^{\top}(\nabla v_n(\theta_n^{*})-\nabla V)}_{L^{\infty}(\Omega_{\frac{\delta}{2}};\R^{m})}=0$, where $a$ is given by \eqref{cor:Stab:c2:a} with $g_2=V$.
   Using \Cref{cor:Stab:c2} we know that for all $y_0\in\omega$ the tuple $(T_n,u_{v_n(\theta_n^*)},\frac{\delta}{2},y_0)$ satisfies \Cref{hypo:stab} with $\delta$ replaced by $\frac{\delta}{2}$ for $n$ large enough. With \Cref{lemma:LipschitzContCalV} we conclude that
   $$ \lim_{n\to\infty}\mathcal{V}_{T_n}(\cdot,u_{v_n(\theta_n^*)})=V\mbox{ in }L^{\infty}(\omega)$$
   and thus the claim for \Cref{Problem:MachLearningProblem2} is verified.

   Next we turn now our attention to \Cref{Problem:MachLearningProblem1}. Let $\theta_n^{*}$ be {{an}} optimal solution for \Cref{Problem:MachLearningProblem1} with $S=S_n$ and $T=T_n$ as before.  Since $D_{y_0}y^{*}(t,y_0)$ is invertible for all $(t,y_0)\in [0,T]\times \overline{\omega}$, by using the change of variables $z=y^*(t;y_0)$  as in \Cref{rem:regressionproblems} we obtain
      $$  \int_{\omega}\int_{0}^{T} |B^{\top}(\nabla v(y^{*}(t;y_0)) -\nabla V(y^{*}(t;y_0)))|^2dtdy_0=\int_{\Gamma_T}|B^{\top}(\nabla v(z)-\nabla V(z))|^{2} g(z)dz$$
      Moreover, we know that $g$ is strictly positive in $\Gamma_T$ and therefore it  is equivalent to the Lebesgue measure in $\Gamma_T$. From this we can continue with the proof as in the case of \Cref{Problem:MachLearningProblem2} to obtain the desired convergence.

\end{proof}

\begin{rem} \label{rem:conn1}

Here  we informally provide  connections between  \Cref{Problem:LearningProblemOpt} and \Cref{Problem:MachLearningProblem1} which are obtained
through their optimality conditions. Let us consider a time horizon $T\in (0,\infty)$ and a setting $S=((\Theta,\norm{\cdot}_{\Theta}),\mathcal{P},\alpha,v)$. The differentiability of the function $\mathcal{V}_{T}(y_0,u)$ for $y_0\in\Omega$ and $u\in C^{2}(\Omega)$ with respect to both arguments comes from \cite{KuVa2} (see also the Appendix).  Using this and \eqref{def:mathcalVphi} we get  the optimality condition of \Cref{Problem:LearningProblemOpt} as
\beq
\begin{array}{l}
\displaystyle\frac{1}{\beta|\omega|}\!\int_{\omega}\int_{0}^{T}\!\!\big(\nabla_{y} v(\theta)(y(t;y_0,u_{v(\theta)}))-\nabla\!_{y}\!\mathcal{V}_{T}(y(t;y_0,u_{v(\theta)}),u_{v(\theta)}\big)^\top\!\! \\[1.5ex] \displaystyle
\qquad \qquad BB^{\top}\nabla \psi(y(t;y_0,u_{v(\theta)})) dtdy_0
+\alpha D\mathcal{P}(\theta)\cdot\vartheta=0,
\label{LearningProblemOpt:optimalityCond}
\end{array}
\eeq
where $\psi=\left( D_{\theta}v\cdot\vartheta \right)$, and $\vartheta\in \Theta$ is arbitrary. We note that \eqref{LearningProblemOpt:optimalityCond} also arises as the optimality conditions of the regression problem given by minimizing over $\theta\in\Theta$ the following expression
\begin{equation*}
\frac{1}
{2\beta |\omega|}\int_{\omega}\int_{0}^{T}|B^{\top}(\nabla_y v(\vartheta)(y(t;y_0,u_{v(\theta)}))\!-\!\nabla\!_{y}\mathcal{V}_{T}(y(t;y_0,u_{v(\theta)}),u_{v(\theta)}))|^2dtdy_0+ \alpha\mathcal{P}(\vartheta)
\end{equation*}

It is noteworthy to observe that this problem is similar to the problem arising in \Cref{Problem:MachLearningProblem1}, replacing $V$ by $\mathcal{V}_{T}$ and $y^{*}$ by $y(\cdot;\cdot,u_{v(\vartheta)})$. This implies that if $\theta^{*}\in\Theta$ is an optimal solution of \Cref{Problem:MachLearningProblem1} and the trajectories $y(\cdot;y_0,u_{v(\theta^{*})})$ are close to $y^{*}(\cdot;\cdot)$ in $W^{1,\infty}((0,T)\times\omega;\R^{d})$, then $v(\theta^{*})$ satisfies approximately  \eqref{LearningProblemOpt:optimalityCond}. This holds in the opposite direction as well, that is, if $\theta^{*}$ is an optimal solution of \Cref{Problem:LearningProblemOpt} then it satisfies approximately the optimality condition associated to \Cref{Problem:MachLearningProblem1}.

\end{rem}
\begin{rem}\label{rem:conn2}
 There is still another important consequence of the optimality conditions of \Cref{Problem:LearningProblemOpt}. In \cite{KuVa} it was proved that for finite horizon problems the optimality condition of this method becomes sufficient if the value function is $C^2$. In the following we will see that a similar situation holds for \Cref{Problem:LearningProblemOpt}. Let us consider $T$ tending to infinity. We formally set  $\mathcal{V}_{\infty}$ as the limit of $\mathcal{V}_{T}$ and  $\mathcal{O}_{\infty}$ as the superior limit of $\mathcal{O}_{T}$ as $T$ tends to infinity.

Let $\theta^{*}\in\mathcal{O}_{\infty}$ be such that it satisfies \eqref{LearningProblemOpt:optimalityCond} and set $y(t;y_0)=y(t;y_0,u_{v(\theta^{*})})$ for $(t,y_0)\in(0,\infty)\times \omega$ . For $T>0$, by \eqref{def:adjoint:CalV:v} and \eqref{def:mathcalVGrad} we deduce that
\begin{equation*}
\begin{array}{l}
     \displaystyle \ell(y(t;y_0))+\frac{1}{2\beta}|B^{\top}\nabla v(\theta^*)(y(t;y_0))|^{2}+\nabla\mathcal{V}_{T}(y(t;y_0))^{\top} \\[1.5ex]
     \qquad \qquad \quad \cdot(f(y(t;y_0))-\frac{1}{\beta}BB^{\top}\nabla v(\theta^*)(y(t;y_0)))
     \displaystyle\\[1.5ex]
     = \ell(y(T;y_0))+\frac{1}{2\beta}|B^{\top}\nabla v(\theta^*)(y(T;y_0))|^{2}.
\end{array}
\end{equation*}
 Further assuming that
\beq \lim_{t\to\infty}\ell(y(t;y_0))+\frac{1}{2\beta}|B^{\top}\nabla v(\theta^*)(y(t;y_0))|^{2}=0,
\label{OptCond:stab}\eeq
and letting $T$ tend to infinity we obtain
\begin{align*}
\ell(y(t;y_0))+\frac{1}{2\beta}|B^{\top}\nabla v(\theta^*)(y(t;y_0))|^{2}+\nabla\mathcal{V}_\infty(y(t;y_0),u_{v(\theta^*)})^{\top} \\[1.5ex]
\qquad \cdot(f(y(t;y_0))-\frac{1}{\beta}BB^{\top}\nabla v(\theta^*)(y(t;y_0)))=0.\label{OptCond:H1}
\end{align*}
If there exists $\hat{\vartheta}\in \Theta$ such that $D_\theta v(\theta^{*})\cdot\hat{\vartheta}= v(\theta^*)-\mathcal{V}_\infty(\cdot,u_{v(\theta^{*})})$, then plugging $\hat{\vartheta}$ in \eqref{LearningProblemOpt:optimalityCond} leads to
\beq B^{\top}\nabla v(\theta^{*})(y(t;y_0))=B^{\top}\nabla \mathcal{V}_{\infty}(y(t;y_0),u_{v(\theta^*)}) \mbox{ for all }(t,y_0) \in (0,\infty)\times\omega.\label{OptCond:ControlEq}\eeq
 We note that if $D_\theta v(\theta^{*})$ is surjective then the existence of such a $\hat{\vartheta}$ is equivalent to  assuming that $\mathcal{V}_\infty(\cdot,u_{v(\theta^{*})})\in\Theta$. Combining \eqref{OptCond:ControlEq} and \eqref{OptCond:H1} we obtain
\beq
\begin{array}l
\ell(y(t;y_0))+\frac{1}{2\beta}|B^{\top}\nabla \mathcal{V}_{\infty}(y(t;y_0),u_{v(\theta^*)})|^{2}\\[1.5ex]
\quad +\nabla\mathcal{V}_\infty(y(t;y_0),u_{v(\theta^*)})^{\top}\cdot(f(y(t;y_0))-\frac{1}{\beta}BB^{\top}\nabla \mathcal{V}_{\infty}(y(t;y_0),u_{v(\theta^*)}))=0,\label{OptCond:H2}
\end{array}
\eeq
which implies that $\mathcal{V}_{\infty}(\cdot,u_{v(\theta^*)})$ satisfies equation \eqref{HJBeq}  on the set $\{ y(t;y_0) \,|\, (t,y_0) \in (0,\infty)\times \omega\}$.
 Further by the verification theorem (see
Theorem I.7.1 in \cite{FlemSon}), if $\omega$ is large enough in the sense  that  $\{ y(t;y_0) \,|\, (t,y_0) \in (0,\infty)\times \omega\} \supset \{ y^*(t;y_0) \,|\, (t,y_0) \in (0,\infty)\times \omega\}$, where $y^*$ is optimal for  \eqref{ControlProblem},
then  $\mathcal{V}_{\infty}(\cdot,u_{v(\theta^{*})})=V$. In this case $u_{v(\theta^{*})}$ provides an optimal feedback law for \eqref{ControlProblem} and it is an optimal solution of \Cref{Problem:LearningProblemOpt}.
\end{rem}

\section{Obstacle problem}
\label{sec:example}
In this section we introduce a parameterized family of control problems with different degrees of smoothness of the value function depending on a penalty parameter. This example was already introduced in \cite{KuVa2}, for that reason we will follow the same definition.

For $\gamma\in [0,\infty)$, we consider the  control problem
\beq \min_{\begin{array}{c}
     u\in L^{2}((0,\infty);\R^{2}),\\
      y'=u,\ y(0)=y_0,
\end{array}} \int_{0}^{\infty} \ell_\gamma(y(t))dt+\frac{\beta}{2}\int_{0}^{\infty}|u(t)|^2 dt,
\label{Example:controlproblem}\eeq
where  the running cost for the state variable is given by
\begin{equation*}
\ell_\gamma(y) = \frac{1}{2}|y|^2\left(1+\gamma\psi\left(\frac{|y-z|}{\sigma}\right)\right),
\text{ with  }
\psi(s)=\left\{\begin{array}{ll}
    \exp\left(-\frac{1}{1-s^2}\right) & \mbox{ if }|s|<1 \\
     0& \mbox{ if }|s|\geq 1,
\end{array}\right.
\end{equation*}
and  $z\in \R^{2}$ satisfies $z_{1}< -\sigma$, $z_2=0$,  for some $\sigma >0$.

 The value function of this problem is denoted  by $V_\gamma$. In \cite{KuVa2},  local Lipschitz regularity of $V_\gamma$ for all $\gamma>0$ and its lack of differentiability for $\gamma$ sufficiently large were proved. Here we will prove that for $\gamma$ close to $0$ the function $V_{\gamma}$ is in $C^{\infty}(\R^2)$. For this purpose we define
\beq \gamma_s=\frac{1}{\sup_{y\in B(z,\sigma) }\left\{2|y|\cdot |\nabla \tilde{\psi} (y)|+\frac{1}{2}|y|\cdot|\nabla ^2\tilde{\psi} (y)|\right\}} \eeq
 with $\tilde{\psi}(y)=\psi\left(\frac{|y-z|}{\sigma}\right)$.
 \begin{prop}
     \label{prop:smoothness} Let $\gamma\in [0,\gamma_s)$. The value function $V_{\gamma}$ of \eqref{Example:controlproblem} is $C^\infty(\R^2)$.
 \end{prop}
  \begin{proof}
     By \cite{KuVa2} we already know that for all $y_0\in\R^{d}$ and $\gamma\in[0,\infty)$ problem \eqref{Example:controlproblem} has at least one solution and  $V_{\gamma}$ is locally Lipschitz continuous. Additionally, we observe that
     $$ \nabla^{2}\ell_{\gamma}(y)=I(1+\gamma\tilde{\psi}(y))+2\gamma y\otimes \nabla\tilde{\psi}(y)+\gamma\frac{|y|^2}{2} \nabla ^2\tilde{\psi}(y)$$
     where $I$ is the $2\times 2$ identity matrix. Multiplying both sides of the previous equality by $x\in \R^{2}\setminus{(0,0)}$ we see that
     $$ x^{\top}\cdot \nabla^{2}\ell_{\gamma}(y) \cdot x\geq \left((1+\gamma\tilde{\psi}(y))-2\gamma|y|\cdot |\nabla \tilde{\psi}(y)|-\gamma\frac{|y|^2}{2}\cdot|\nabla^2 \tilde{\psi}(y)|\right)|x|^2.$$
     Using that the support of $ \tilde{\psi}$ is $B(z,\sigma)$ in the previous inequality, we obtain for $\gamma\in [0,\gamma_s)$ that
     \beq x^{\top}\cdot \nabla^{2}\ell_{\gamma}(y) \cdot x\geq \left(1-\gamma \sup_{y\in B(z,\sigma) }\left\{2|y|\cdot |\nabla \tilde{\psi} (y)|+\frac{1}{2}|y|\cdot|\nabla ^2\tilde{\psi} (y)|\right\}\right)|x|^2>0,\label{theo:smoothness:positivedef}\eeq
     which implies that $\ell_{\gamma}$ is a strictly convex function. Due to the linearity of the state equation of \eqref{Example:controlproblem}, we also get that the objective function of \eqref{Example:controlproblem} is convex and consequently the optimal control of \eqref{Example:controlproblem} is unique for $\gamma\in [0,\gamma_s)$.

     In the following, we will use \eqref{theo:smoothness:positivedef} to prove that $V_{\gamma}$ is $C^2(\R^2)$ by means of the implicit function theorem. For this purpose we define the function $\Phi: X\times \R^2\mapsto Y$ where $X=H^{1}((0,\infty);\R^2)\times L^{2}((0,\infty);\R^2)\times H^{1}((0,\infty);\R^2)$, and $Y=L^{2}((0,\infty);\R^2)\times L^{2}((0,\infty);\R^2)\times\R^2\times L^{2}((0,\infty);\R^2)$  by
     \beq
     \Phi(y,u,p,y_0)=\left(\begin{array}{l}
          -p'+\nabla \ell_{\gamma}(y) \\
          y'-u\\
          y(0)-y_0\\
          \beta u- p
     \end{array}\right).
     \label{theo:smoothness:Phidef}
     \eeq
      We note that $\Phi(y,u,p,y_0)=0$ is  the optimality condition for \eqref{Example:controlproblem}. Clearly the function $\Phi$ is of class $C^{\infty}$ and its differential with respect to its first three variables is given by
     $$D_{y,u,p}(y,u,p,y_0)\cdot (\delta_y,\delta_u,\delta_p)=\left(\begin{array}{l}
          -\delta_p'+\nabla^2 \ell_{\gamma}(y)\delta_y \\
          \delta_y'-\delta_u\\
          \delta_y(0)\\
          \beta\delta_u- \delta_p
     \end{array}\right)$$
     for  $((y,u,p),y_0)\in X\times \R^2 $ and
      $(\delta_y,\delta_u,\delta_p)\in X$. In order to use the implicit function theorem we need to prove that for $u^{*}$ the optimal solution of \eqref{Example:controlproblem}, $y^{*}$ its optimal trajectory and $p^{*}$ the respective adjoint state, the partial sub-differential $D_{y,u,p}(y^{*},u^{*},p^{*})$ is invertible , i.e., for all $(z_p,z_y,z_{y_0},z_u)\in Y$ there exist only one  $(\delta_y,\delta_u,\delta_p)\in X$ such that
      $$ D_{y,u,p}(y^*,u^*,p^*)\cdot (\delta_y,\delta_u,\delta_p)= (z_p,z_y,z_{y_0},z_u).$$
      We note that this equation is equivalent to the optimality conditions of the following linear-quadratic problem
      \beq
      \left\{
      \begin{array}l
      \min
      \int_{0}^{\infty}\left(\frac{1}{2}\delta_y^{\top}(t)\nabla^2\ell_{\gamma}(y^*(t))\delta_y(t)-z_p(t)\cdot\delta_y(t)
      +\frac{\beta}{2}|\delta_u(t)|^2-\delta_u(t)\cdot z_u(t)\right)dt,\\[1.5ex]
      \text{where } \delta_u\in L^{2}((0,\infty);\R^2), \text{and }
      \delta_y'=\delta_u+z_y,\ \delta_y(0)=z_{y_0}.
      \label{Example:LinearizedProblem}
      \end{array} \right.
      \eeq  
      Consider $\tilde{u}$ the optimal solution of \eqref{Example:controlproblem} with initial condition $z_{y_0}$ and $\tilde{y}$ the corresponding trajectory. We note that using $\delta_u=\tilde{u}-z_y $ we obtain that $\delta_y=\tilde{y}$ satisfies $\delta_y'-\delta_u=z_y$ and $$\left|\int_{0}^{\infty}\left(\frac{1}{2}\delta_y^{\top}(t)\nabla^2\ell_{\gamma}(y(t))\delta_y(t)-z_p(t)\cdot\delta_y(t)+\frac{\beta}{2}|\delta_u(t)|^2-\delta_u(t)\cdot z_u(t)\right)dt\right|<+\infty.$$
     This together with the positive definiteness of $\nabla^2\ell_{\gamma}$ and weak lower semi-continuity of the objective function  imply that \eqref{Example:LinearizedProblem} has a unique optimal solution. Then by the implicit function theorem there exists an open neighborhood $\Gamma\subset\R^2$ of $y_0$ and functions  $\hat{u}:\Gamma\mapsto L^{2}((0,\infty);\R^2)$, $\hat{y}:\Gamma\mapsto H^{1}((0,\infty);\R^2)$ and $\hat{p}:\Gamma\mapsto H^{1}((0,\infty);\R^2)$ of class $C^\infty$, such that for all $z_0\in \Gamma$ we have that  $\Phi(y(z_0),u(z_0),p(z_0),z_0)=0 $.Since for $\gamma\in [0,\gamma_s)$  problem \eqref{Example:controlproblem} is strictly convex, we have that $u(z_0)$ is  its optimal solution, $y(z_0)$ its optimal trajectory, and $p(z_0)$ the corresponding adjoint state. Hence, $V_{\gamma}(z_0)=\int_{0}^{\infty}\left(\ell_{\gamma}(y(z_0))+\frac{\beta}{2}|u(z_0)|^2\right) dt$ and by the chain rule we obtain that $V_{\gamma}$ is $C^{\infty}$ in $\Gamma$, Since $y_0$ is arbitrary we deduce that $V_{\gamma}$ is $C^{\infty}(\R^2)$. 
 \end{proof}
 As  mentioned before, in \cite{KuVa2} it was proved that for $\gamma$ large enough the value function $V_\gamma$ is non differentiable. Specifically, for each such $\gamma$  there exists $x_{\gamma}$ such that $V_{\gamma}$ is non-differentiable in $\{(x,0)\in \R^2:x>x_{\gamma}\}$. However, it was also proved in \cite{KuVa2} that for all $\gamma>0$ \Cref{hypo:stab} and \Cref{hypo:approx:vf} hold. Consequently, by using a sequence of setting satisfying \Cref{hypo:approx:c2} we obtain by means of \Cref{theo:conv:LearningProblemOpt} that convergence in the sense of \eqref{ConvergenceDef} holds for  \Cref{Problem:LearningProblemOpt}. On the other hand, for Methods \ref{Problem:MachLearningProblem1}  and \ref{Problem:MachLearningProblem2}  and  $\gamma\in [0,\gamma_s)$  we obtain that
  \Cref{hypo:approx:restricc2} holds provided the sequence of settings is chosen appropriately.

This together with  \Cref{prop:smoothness} allows us to utilize  problem \eqref{Example:controlproblem} to  numerically investigate the behavior of Methods \ref{Problem:LearningProblemOpt} - \ref{Problem:MachLearningProblem2} for different choices of $\gamma$, which  account for different degrees of smoothness of $V_{\gamma}$.

We close this section by establishing another important property of the value function.

 \begin{prop}
     \label{prop:scc}
The value function associated to \eqref{Example:controlproblem} is semiconcave on every compact subset of $\R^2$.
\end{prop}
\begin{proof}
Let $C\subset \R^2$ be compact and denote by $(y^*(\cdot;y_0, u^*(\cdot;y_0)))$ the optimal state-control pairs with $y_0\in C$.
Since the optimal trajectories associated to \eqref{Example:controlproblem} are globally exponentially stable there exists $\hat T$ such that $y^*(t;y_0) <|z| - \sigma$  for all $ (t,y_0) \in [\hat T, \infty) \times C$,  and consequently $V_\gamma(y_0) = V_0(y_0)=\frac{\sqrt{\beta}}{2} |y_0|^2$, \cite[Lemma 9.2, Lemma 9.3]{KuVa2}. Here $ B(0,|z| - \sigma)$
denotes the open ball with center at the origin and radius $|z| - \sigma$.
Further $|u^*(t;y_0)| = \frac{2}{\beta} \ell_\gamma(y^*(t;y_0))$, \cite[Proposition 9.1]{KuVa2}, and thus for a constant $K$ we have $|u^*(t;y_0)| \le K$,  for all $(t,y_0)\in [0,\infty)\times C$.

 Let us consider
 \beq \min_{\begin{array}{c}
     u\in L^{2}((0,\hat T);\R^{2}), \, |u(t)|\le K+1\\
      y'=u,\ y(0)=y_0 \in C
\end{array}} \int_{0}^{\hat T} \ell_\gamma(y(t))dt+\frac{\beta}{2}\int_{0}^{\hat T}|u(t)|^2 dt + \hat V_0(y(\hat T,y_0)),
\label{Example:aux_controlproblem}\eeq
where  $\hat V_0$ denotes the extension of $V_0$   outside of $B(0,|z| - \sigma)$ by $\hat V_0(y) = \frac{\sqrt{\beta}}{2} |y|^2$.

For  $y_0\in C$ this problem is equivalent to \eqref{Example:controlproblem}. Here we use Bellman's principle, the fact that
$y^*(\hat T;y_0) \in B(0,|z| - \sigma)$, and  that
 $|u^*|_{L^\infty(0,\infty)} \le K$, so that the control constraint in \eqref{Example:aux_controlproblem} is not active at the solutions. The value function restricted to $C$ coincides for these two problems. We next utilize that  \eqref{Example:aux_controlproblem} can be considered as finite horizon problem with terminal penalty given by $\hat V_0$. It then follows from \cite[Theorem 7.4.11]{Cannarsa2004} that the value function for \eqref{Example:aux_controlproblem} is semiconcave on $C$, provided that the assumptions of that  theorem are satisfied. In this respect it suffices to recall that any function which is $C^1(A)$ in an open set $A$ containing $C$ is semiconcave on $C$, \cite[Proposition 2.1.2]{Cannarsa2004}. This is directly applicable for $\ell\in C^1(\R^2)$ as well as $\hat V_0\in C^\infty(\R^2)$. The remaining assumptions on the dynamical system and the running cost are trivially satisfied.
\end{proof}

\section{Numerical experiments}

Numerical tests for Methods \ref{Problem:LearningProblemOpt} and \ref{Problem:MachLearningProblem1} including high-dimensional problems  are available in the literature, we refer to eg. \cite{KW1, Ruthotto, AzKaKK, KuVa2, KuVaWal, Onken, Onken2}. Thus the  strength and validity of them for practical realisation has already been established.
Here we make an effort to compare these
methods with respect to the influence of the regularity of the value function. For this purpose, we employ equation
\eqref{Example:controlproblem} and  recall that the  parameter $\gamma$  acts as a control for the regularity of $V_\gamma$. Thus we are interested whether the performance of  these methods significantly differs with respect to the regularity of the value function. 
 It would be of interest to also compare these methods for high dimensional problems. This, however, is not within the scope of this research.

We  consider   problem \ref{Example:controlproblem} with $z_1=-2$, $\sigma=1$, and set $\omega=(-5,5)\times (-2,2)$. Since the dimensionality  is low,  a regular grid  is used to  approximate the integrals appearing in Methods \ref{Problem:LearningProblemOpt} - \ref{Problem:MachLearningProblem2}. For high dimensional problems this is not feasible and other discretization method should be used instead, as for example Monte-Carlo integration, which is the usual choice for machine learning methods and has been used in \cite{KuVaWal, KuVa}. We
 report on the behavior of the obtained feedback laws for each of the  methods with $\gamma\in\{10^{-3},10^{-2},10^{-1},1,10,10^2,10^3\}$.

In the following we further describe the setting of the methods. For the Banach spaces $\Theta$ and the parametrization we follow  \Cref{rem:example:approx}. Thus the setting depends on two parameters, which are the order $k$ and the degree $n$. Here we consider $k\in\{1,2\}$ and $n\in \{10,20,30,40\}$. Recall that the choice of $k$ connects the smoothness of the value function.
 The penalty function of the setting is $\mathcal{P}(\theta)=\frac{1}{2}\left(\frac{1}{2}|\theta|_2^2+|\theta|_1\right)$ and the our tests were carried out with penalty parameter  $\alpha$ chosen from the set $\{10^{-1},10^{-3},10^{-5},10^{-7},10^{-9}\}$.
 In all the  cases we consider the time horizon with $T=1$.  

The integrals over $\omega$ appearing in the learning problems are approximated over  a regular grid of size $16\times 16$. We call these points the \textbf{training set}. In order to study the performance of the feedback produced by each method we also consider a finer regular grid of size $32\times 32$, which is the \textbf{test set}. In both the training and test sets the performance of each method is evaluated  by the following error measures:
\beq
NMAE_V=\sum_{i=1}^{N}|V(y_0^i)-\mathcal{V}_{T}(y_0^i,v(\theta))|\cdot \frac{1}{\sum_{i=1}^{N}V(y_0^i)},
\eeq
\beq
NMRSE_c=\sum_{i=1}^{N}\norm{u_i^{*}-\hat{u}_i}_{L^{2}((0,T);\R^2)}\cdot \frac{1}{\sum_{i=1}^{N}\norm{u_i^{*}}_{L^{2}((0,T);\R^2)}},
\eeq
where $\{y_0^{i}\}_{i=1}^{n}$ is a set of initial conditions which in this case could be either the training or the test sets, $\{u_i^{*}\}_{i=1}^{N}$ are the optimal controls and $\{\hat{u}_i\}_{i=1}^{N}$ are the controls obtained by applying the respective feedback. For obtaining a reference value for the optimal value function $V$ and the optimal controls for  the chosen  initial conditions we numerically  solve the discretized open loop control problem. For this purpose the $ODEs$ and the integral of the running cost are discretized by an explicit Euler scheme with time step  $\frac{1}{400}$. In accordance  with Theorems \ref{theo:conv:LearningProblemOpt} and \ref{theo:conv:LearningProblemReg} we expect that for small values of  $\gamma$ all three approaches will achieve a good performance since in this case $V_{\gamma}$ is smooth. On the other hand,  $V_{\gamma}$ is non-smooth  for large values of  $\gamma$,   and only Theorem \ref{theo:conv:LearningProblemOpt}  is applicable, which  implies that we only have a guarantee of good performance for Method \ref{Problem:LearningProblemOpt}.

For solving the learning problems we employ a proximal gradient method together with the Barzilai-Borwein step size and a non-linear backtracking line search. This algorithm is described in \cite{KuVaWal} and \cite{KuVa} in the context of \Cref{Problem:LearningProblemOpt}.
Following the notation of Algorithm 1 in \cite{KuVa} we use $k=3$, $tol=10^{-6}$, $\kappa=10^{-3}$ and $\xi=\frac{1}{2}$,  with  the same  tolerances measures as in \cite{KuVa}. Concerning initialization,
for each setting we first solve with $\gamma=10^{-3}$ using as initial guess $\theta^0=0$, and subsequently  for $\gamma>10^{-3}$  the solution corresponding to the previous $\gamma$ value is used  as initial guess. The penalty parameter $\alpha$ was chosen such that it  achieve approximately the best $NMAE_V$ in the training set. The reported results were obtained on a computer with a Xeon E5-2630v3 (2x8Core 2,4Ghz) processor and 8 GB of RAM memory.

\begin{figure}[h!]
         \centering
\foreach \alphai/\alphaj in {0/10^{-3},3/1,6/10^3}{
\begin{subfigure}[b]{0.45\textwidth}
\centering
 \includegraphics[width=\textwidth]{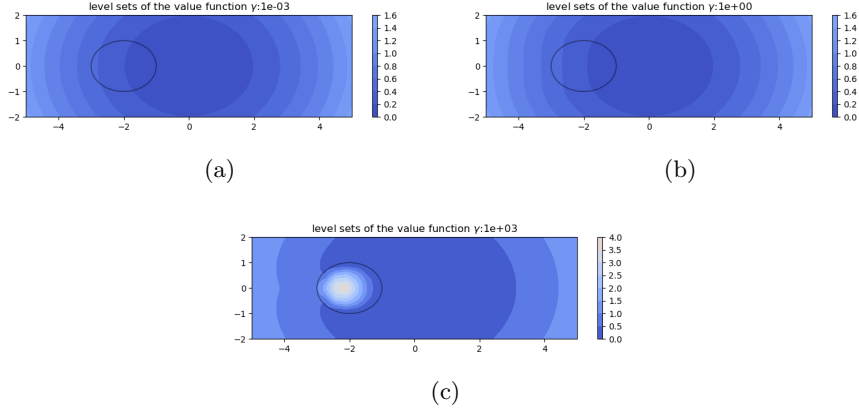}
         \caption{}
         \label{levelSets:\alphai}
\end{subfigure}
}
\caption{Level sets of $V_{\gamma}$, the color of the level sets are in logarithmic scale.}
\label{fig:levelSets}
\end{figure}

We start the discussion of the experiments with the level sets of the value function for selected  values of regularisation parameters $\gamma=10^{-3},0,10^3$.  This is relevant to understand the distribution of the errors for each method which will be reported later on. In Figure \ref{fig:levelSets} the level sets for selected $\gamma$-values  together with the boundary of the obstacle are presented. We observe that on  the right hand side of the $y-$axis the level sets do not change with $\gamma$. This is concordant with the fact that in this region the obstacle has no effect. On the other hand, on the left hand side we observe that in a neighborhood of the $x-$axis the level sets vary from being convex to non-convex. This is consistent  with the fact that on this axis the value function is non-differentiable for large $\gamma$. Further, for $\gamma=10^3$ the value function has a local maximum close to the center of the obstacle. Indeed, since $V_{\gamma}$ converges to $V_{\infty}$ it is expected that $V_{\gamma}$ blows up inside the obstacle as $\gamma$ tends to infinity.
\begin{figure}[h!]
\centering
\begin{subfigure}[b]{0.45\textwidth}
\centering
\includegraphics[width=\textwidth]{"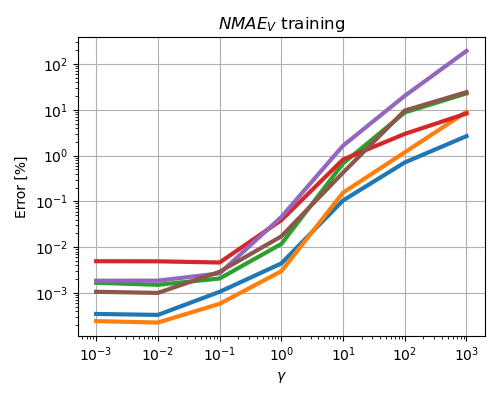"}
\caption{}
\label{fig:NMAE:training}%
\end{subfigure}
\begin{subfigure}[b]{0.45\textwidth}
\centering
\includegraphics[width=\textwidth]{"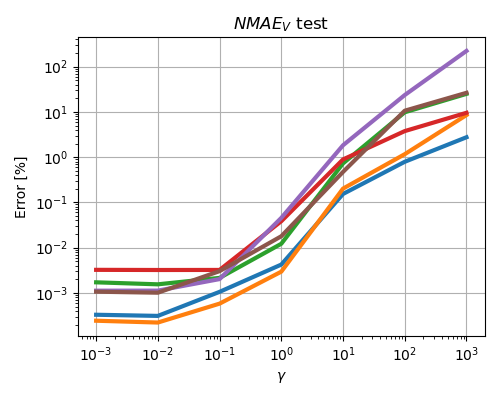"}
\caption{}
\label{fig:NMAE:test}%
\end{subfigure}
\begin{subfigure}[b]{0.45\textwidth}
\centering
\includegraphics[width=\textwidth]{"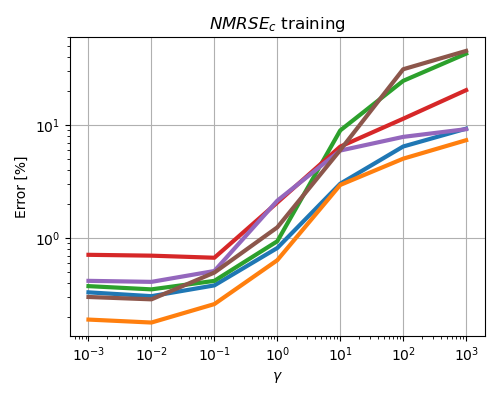"}
\caption{}
\label{fig:NMRS:training}%
\end{subfigure}
\begin{subfigure}[b]{0.45\textwidth}
\centering
\includegraphics[width=\textwidth]{"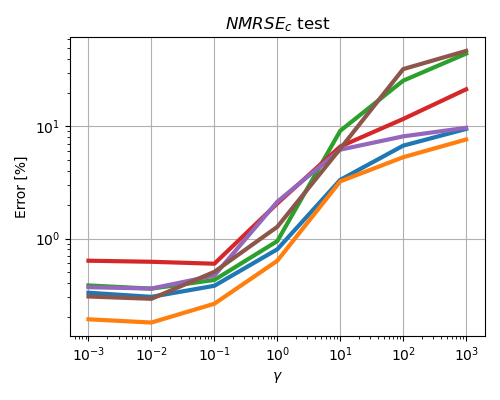"}
\caption{}
\label{fig:NMRS:test}%
\end{subfigure}
\begin{subfigure}[b]{0.45\textwidth}
\centering
 \includegraphics[width=\textwidth]{"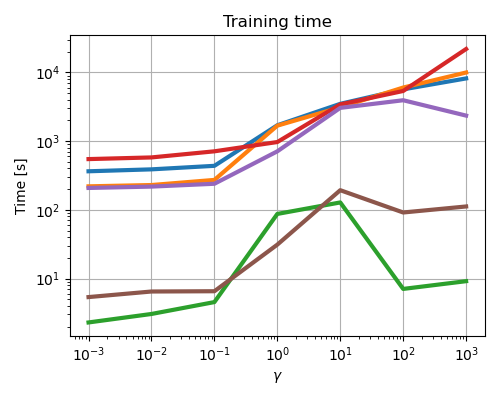"}
\caption{}
\label{fig:time}%
\end{subfigure}
\begin{subfigure}[b]{0.45\textwidth}
\centering
 \includegraphics[width=\textwidth]{"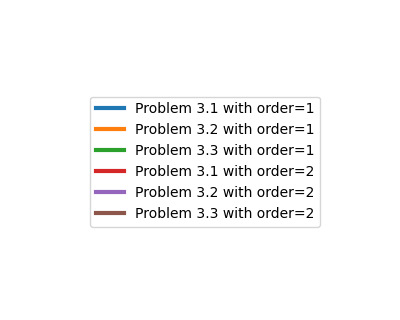"}
\end{subfigure}
\caption{Errors and training times. In all the figures both axes are in logarithmic scale.}
\label{fig:errors}
\end{figure}

\begin{figure}
\centering
\begin{subfigure}[b]{0.45\textwidth}
\centering
 \includegraphics[width=\textwidth]{"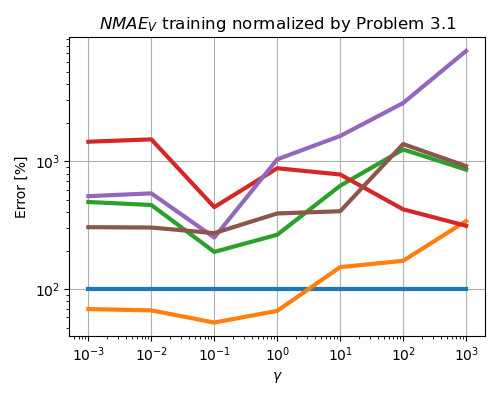"}
 \caption{}
 \label{fig:NMAE:Rel:training}
\end{subfigure}
\begin{subfigure}[b]{0.45\textwidth}
\centering
 \includegraphics[width=\textwidth]{"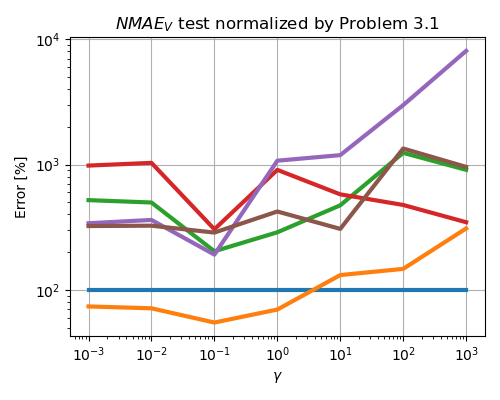"}
\caption{}
\label{fig:NMAE:Rel:test}
\end{subfigure}
\begin{subfigure}[b]{0.45\textwidth}
\centering
 \includegraphics[width=\textwidth]{"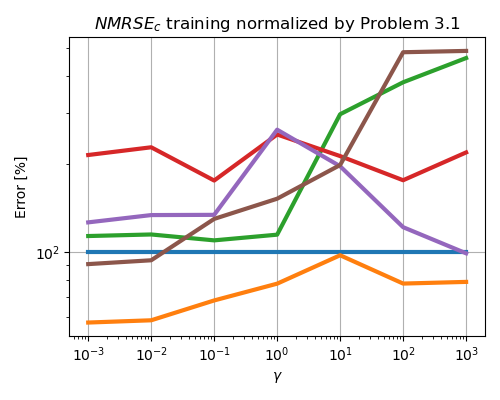"}
 \caption{}
 \label{fig:NMRS:Rel:training}
\end{subfigure}
\begin{subfigure}[b]{0.45\textwidth}
\centering
 \includegraphics[width=\textwidth]{"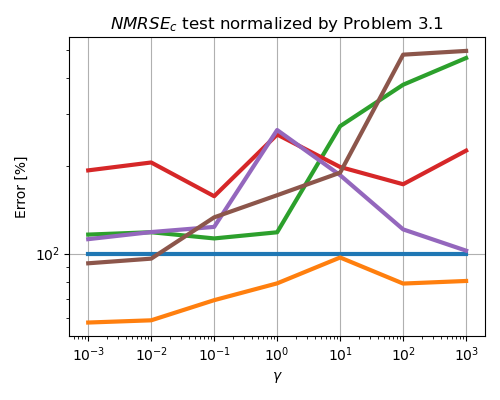"}
\caption{}
\label{fig:NMRS:Rel:test}
\end{subfigure}
\begin{subfigure}[b]{1\textwidth}
\centering
 \includegraphics[width=\textwidth]{"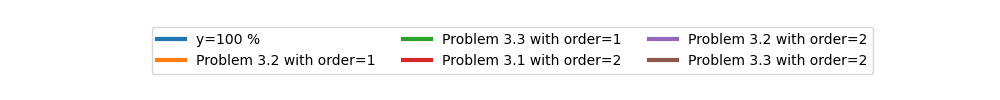"}
 \end{subfigure}
\caption{Errors normalized by the error of Method 3.1. In all the figures both axes are in logarithmic scale.}
\label{fig:errorsRel}

\end{figure}

To compare the performance graphically,  in Figure \ref{fig:errors} the errors and the training time for each method are depicted. We observe that typically $k=1$  achieves a better performance than $k=2$, especially for $\gamma\geq 10$. Additionally, we note that the error increases with $\gamma$. Methods \ref{Problem:LearningProblemOpt} and  \ref{Problem:MachLearningProblem1} with $k=1$ achieved the best performance in both error measures. To take a closer look we turn our attention to Figure \ref{fig:errorsRel} where the errors are normalized by dividing them by the error obtained by using Method \ref{Problem:LearningProblemOpt} with k=1. From this we see that for $\gamma\leq 1$ the method with the best performance in terms of $NMAE_V$ and $NMRSE_c$ control is Method \ref{Problem:MachLearningProblem1} with $k=1$. In contrast, for $\gamma>1$, Method \ref{Problem:LearningProblemOpt} with $k=1$ has the best $NMAE_V$ and moreover the distance with respect to  Method \ref{Problem:MachLearningProblem1} increases. We point out that also for $\gamma>1$ the fact that Method \ref{Problem:MachLearningProblem1} has best $NMRSE_c$ control does not imply the same with respect to $NMAE_V$. Concerning the efficiency of the methods  in \Cref{fig:errors}(\subref {fig:time}) we provide the trainings time for each experiment, that is, the time that it  takes to achieve the stooping criterion. In this figure we see that Method \ref{Problem:MachLearningProblem2} is the one with the shortest training time. For Methods \ref{Problem:LearningProblemOpt} and \ref{Problem:MachLearningProblem1} the training times tend to increase with $\gamma$. This can be  a consequence of decreasing regularity  and the fact that the value function blows up inside the obstacle as $\gamma$ tends to infinity.

\foreach \alphai/\alphaj in {6/10^3}{
\begin{figure}[h!]
         \centering
\foreach \order in {1,2}{
\foreach \problem [count=\xi] in {opt,tra"_"reg,reg}{
\begin{subfigure}[b]{0.3\textwidth}
\centering
         \includegraphics[width=\textwidth,height=\textwidth]{"Graficos/scatter_"\problem"_"\order"_"\alphai".png"}
         \caption{}
         \label{scatter:\alphai:\xi:\order}
\end{subfigure}
}
}
\hfill
\caption{Scatter plots for $\gamma=\alphaj$.}
\label{scatter:\alphai}
\end{figure}
}

\foreach \alphai/\alphaj in {0/10^{-3},3/1,6/10^3}{
\begin{figure}[h!]
         \centering
\foreach \problem [count=\xi] in {opt,tra"_"reg,reg}{
\foreach \order in {1,2}{
\begin{subfigure}[b]{0.45\textwidth}
\centering
         \includegraphics[width=\textwidth]{"Graficos/error_"\problem"_"\order"_"\alphai".png"}
         \caption{}
         \label{error:\alphai:\xi:\order}
\end{subfigure}
}
}
\hfill
\caption{Error distribution for $\gamma=\alphaj$.}
\label{SpatialError:\alphai}
\end{figure}
}

The previous analysis does not address the spatial distribution of the error. This information is given in  Figures \ref{SpatialError:0}-\ref{SpatialError:6} where the distribution of the error between the value functions $V_{\gamma}$ and $\mathcal{V}_{T}(\cdot,\theta^{*})$ is depicted, with  $\theta^{*}$  the solution obtained by the respective experiment. From Figure \ref{SpatialError:0} we observe that the error is uniformly smaller than $10^{-3}$ for $\gamma\leq 10^{-3}$, and we can report that this is also the case as long as $\gamma \le 1$.   Starting from $\gamma=10^{1}$ we see from Figures \ref{SpatialError:3}-\ref{SpatialError:6} that the error in the region left to the $y-$axis tends to increase with $\gamma$. Moreover, the error is higher for the points close to the $x-$axis which are to the left of the $y-$axis and also inside the obstacle than those outside. As explained  before, this can be attributed to the fact that the value function is non smooth on the $x-axis$  for all $x$ left of the obstacle  and that it blows up inside the obstacle as $\gamma$ tends to infinity.

 Figures \ref{scatter:6} show the scatter plot and the dispersion between $V_{\gamma}$ and $\mathcal{V}_{T}(\cdot,\theta^{*})$ evaluated at the test points for $\gamma=10^3$.
In all the cases the points in the interior of the obstacle tend to be further from the identity line than those outside of the obstacle. For Method  \ref{Problem:LearningProblemOpt} this effect is weaker than for the others.

\section{Conclusions}

In this work convergence in the sense of \eqref{ConvergenceDef}  is analyzed for
 three machine learning approaches. They differ with respect to the assumption on the regularity of the value function.
Further, now that more techniques to compute suboptimal feedback controls guided by the Bellman principle and hence the HJB equation,  become available, it is of interest to start  towards the challenging goal of comparing methods. For this purpose we present a family of control problems in   \Cref{sec:example} indexed by a penalty parameter $\gamma>0$ for which the value function is smooth if $\gamma$ close to 0, and while it is  non-differentiable but semi-concave for $\gamma$ large. This family of control problem was used to compare the performance of the three methods. From these experiments it was observed that for $\gamma$ close to 0, which correspond to the case of smooth value functions, Method
\ref{Problem:MachLearningProblem1} has a better performance than the other approaches, but for larger $\gamma$ Method \ref{Problem:LearningProblemOpt} is the one with a better performance. This is coherent with the convergence results.

Further extensions of the work presented here are possible. For instance, the discretized version of the learning problems was not studied. This is particularly relevant for high-dimensional problems, since in this case the discretization is usually carried out by using Monte-Carlo integration.
Problems with constraints on either the state or the control can be considered by first using an exterior penalty method as was used for the construction of problem \eqref{Example:controlproblem}. Finally, the extension of this problem to higher dimension is of interest, because it could help to study the behavior of the learning problems in the context of high-dimensions  and non-smoothness of the value function.

\appendix
\section{Properties of $\mathcal{V}_T$}
\label{Appendix:CalV}
In this section we provide some important result concerning  properties of $\mathcal{V}_{T}$ needed throughout this work. The proofs can be obtained with standard techniques and are therefore not presented here.

\begin{lemma}
    \label{lemma:LipschitzContCalV}
    Let $T>0$ denote a finite time horizon, and let $y_0\in\omega$, $g_1\in C^{2}(\Omega)$, $g_2\in C^{2}(\Omega)$, and $\delta>0$ be such that $(T,u_{g_1},\delta,y_0)$ and $(T,u_{g_2},\delta,y_0)$ \Cref{hypo:stab} holds. Then we have for all $t\in [0,T]$
    \beq
    |y(t;y_0,u_{g_1})-y(t;y_0,u_{g_2})|\leq \frac{|B|\norm{u_{g_1}-u_{g_2}}_{L^\infty(\Omega_\delta;\R^m)}}{a}\left(e^{ta}-1\right)
    \label{lemma:LipschitzContCalV:state}
    \eeq
    and
    \beq
    \begin{array}{ll}
    \displaystyle |\mathcal{V}_{T}(y_0,u_{g_1})-\mathcal{V}_{T}(y_0,u_{g_2})|\\[1.5ex]
    \displaystyle\leq |B|\big[\frac{\norm{\nabla \ell}_{L^{\infty}(\Omega_\delta);\R^d)}}{a} c_T  +\frac{T\beta}{2}\norm{u_{g_1}
    +u_{g_2}}_{L^{\infty}(\Omega_{\delta};\R^d)}\big]\norm{u_{g_1}-u_{g_2}}_{L^\infty(\Omega_\delta;\R^m)},
    \end{array}
    \label{lemma:LipschitzContCalV:value}
    \eeq
    where
    $$ c_T= (\frac{e^{aT}-1}{a}-T)  \text{ and }a=\norm{Df}_{L^{\infty}(\Omega_\delta;\R^{d\times d})}+\frac{|B|^2}{\beta}\norm{\nabla^2 g_2}_{L^{\infty}(\Omega_{\delta};\R^{d\times d})}.$$
\end{lemma}

\begin{lemma}
    \label{lemma:Differemtiability}
    Let $T>0$ be a time horizon, an let  $y_0\in\omega$, $u\in C^{1}(\Omega)$,  and $\delta>0$ be  such that $(T,u,\delta,y_0)$ \Cref{hypo:stab} holds  for all $y_0\in\omega$. For arbitrary $y_0\in\omega$ let us define the adjoint state $p(\cdot;y_0,u)$ as the solution of the following adjoint equation
    \beq
-p'(t)+\nabla \ell(y(t))+\beta Du(y(t))^{\top}\cdot u(y(t))-(Df(y(t))+BDu(y(t)))^{\top}\cdot p(t)=0 
\label{def:adjoint:CalV:u}
\eeq
for all $t\in (0,T);\ p(T)=0$ ,with $y(t)=y(t;y_0,u)$. Then the function $\mathcal{V}_{T}$ its differentiable with respect to each of its variables and we have
\beq
D_{u}\mathcal{V}_{T}(y_0,u)\cdot\nu=\int_{0}^{T}(\beta u(y(t;y_0,u))-B^{\top}p(t;y_0,u) )^{\top}\cdot \nu(y(t;y_0,u))dt \label{def:mathcalVDu}\eeq
for all $\nu\in C^{1}(\Omega;\R^m)$, and
\beq
\nabla_{y} \mathcal{V}_{T}(y_0,u)=-p(0;y_0,u).
\label{def:mathcalVGrad}
\eeq
Further, for all $t\in (0,T]$ we have
\beq
\nabla_{y} \mathcal{V}_{T}(y(t),u)=-p(0;y(t),u)=-p(t;y_0,u).
\label{def:mathcalVGrad:2}
\eeq
\end{lemma}
\begin{rem}
{\em Under the hypotheses of the previous lemma and assuming $u=u_\phi$ the adjoint equation can be expressed in terms of $\phi$ as
\beq
-p'(t)+\nabla \ell(y(t))-Df^{\top}(y(t)) p(t)+\frac{1}{\beta}\nabla^{2}\phi(y(t))BB^{\top}(\nabla \phi(y(t))+p(t))=0 
\label{def:adjoint:CalV:v}
\eeq
for all $t\in (0,T);\ p(T)=0$. Additionally, using \eqref{def:mathcalVGrad:2} we can express the derivative of $\mathcal{V}_{T}(y_0,u_\phi)$ with respect to $\phi$ as
\begin{equation*}
\begin{array}l
D_{\phi}\left(\mathcal{V}_{T}(y_0,u_\phi)\right)\cdot\psi\\[1.5ex]
\quad =\frac{1}{\beta}\int_{0}^{T}(\nabla \phi(y(t;y_0,u_\phi))- \nabla_{y} \mathcal{V}_T(y(t;y_0,u_\phi)))^{\top} BB^{\top} \nabla \psi(y(t;y_0,u_\phi))\,dt \label{def:mathcalVphi}
\end{array}
\end{equation*}

for $\psi\in C^{2}(\Omega)$.
}
\end{rem}

\bibliographystyle{siamplain}
\bibliography{references}
\end{document}